\documentclass[final]{siamltex}
\usepackage{latexsym, amssymb, amsmath}
\usepackage{graphicx}
\usepackage{mathtools}
\usepackage{float}
\usepackage{mathdots}
\usepackage{color}
\usepackage{enumitem}
\usepackage{comment}

\newtheorem{example}[theorem]{Example}
\let\oldexample\example
\renewcommand{\example}{\oldexample\normalfont}
\newtheorem{remark}[theorem]{Remark}
\let\oldremark\remark
\renewcommand{\remark}{\oldremark\normalfont}
\usepackage{mathdots}

\title{Linearizations for interpolatory bases - a comparison: New families of linearizations.
\thanks{{This research has been funded by the NSF grant DMS-1850663, by the College of Creative Studies donors via The Create Fund, and by private donors Ms. Susie Fitzgerald, and Mr. Manson Jones. }}}
\author{A. Ashkar\thanks{Department of Mathematics, University of California, Santa Barbara, CA 93106, USA ({\tt anthonyashkar@ucsb.edu}){ .}}
\and M. I.  Bueno\thanks{Department of Mathematics and College of Creative Studies,
University of California, Santa Barbara, CA 93106, USA ({\tt mbueno@math.ucsb.edu}){ .}}
 \and R. Kassem \thanks{Department of Mathematics, Duke University, NC 27710, USA ({\tt remy.kassem@duke.edu}){ .}}
  \and D. Mileeva\thanks{Department of Mathematics and College of Creative Studies, University of California, Santa Barbara, CA 93106, USA ({\tt dmileeva@ucsb.edu})}{ .}
\and J. P\'erez \thanks{Department of Mathematical Sciences, University of Montana, Missoula, MT 59812, USA ({\tt javier.perez-alvaro@mso.umt.edu})}{ .}}

\begin{document}
\maketitle

\begin{abstract}
One strategy to solve a nonlinear eigenvalue problem  $T(\lambda)x=0$ is to solve a polynomial eigenvalue problem (PEP) $P(\lambda)x=0$ that approximates the original problem through interpolation.  Then, this PEP is usually solved  by linearization. Because of the polynomial approximation techniques, in this context, $P(\lambda)$ is expressed in a non-monomial basis. The bases used with most frequency are the Chebyshev basis, the Newton basis and the Lagrange basis. Although, there exist already a number of linearizations available in the literature for matrix polynomials expressed in these bases, we introduce new families of linearizations that present the following advantages: 1) they are easy to construct from the matrix coefficients of $P(\lambda)$ when this polynomial is expressed in any of those three bases; 2) their block-structure is given explicitly; 3)  we provide equivalent formulations for all three bases which allows a natural framework for comparison.  We also provide recovery formulas of eigenvectors (when $P(\lambda)$ is regular) and  recovery formulas of minimal bases and minimal indices (when $P(\lambda)$ is singular).  Our ultimate goal is to use these families to compare the numerical behavior of the linearizations associated to the same basis (to select the best one) and with the linearizations associated to the other two bases, to provide recommendations on what basis to use in each context. This comparison will appear in a subsequent paper.
\end{abstract}


\section{Introduction}

Nonlinear eigenvalue problems of the form
\begin{equation}\label{eq:NEP}
T(\lambda)x = 0 \quad \quad \mbox{and} \quad \quad y^TT(\lambda)=0,
\end{equation}
where $T:\Omega\subseteq \mathbb{C}\rightarrow\mathbb{C}^{n\times n}$ is a complex-valued matrix function holomorphic in a complex region $\Omega$, often arise in applications \cite{nonlinear}.
The scalar $\lambda\in\Omega$ is called an \emph{eigenvalue} of $T(\lambda)$, and $x$ and $y$ are  associated \emph{right and left eigenvectors}.

A possible approach for solving the nonlinear eigenvalue problem \eqref{eq:NEP} is to replace $T(\lambda)$ with a matrix polynomial approximation $P(\lambda)$ \cite{chebyshev pep, Lagrange, Lagrange2}.
Such polynomial approximant can be found via interpolation, i.e., for a given set of points $\{x_1,x_2,\cdots,x_{k+1}\}\subset\Omega$, whose elements we call the \emph{nodes}, one replaces $T$ by the unique matrix polynomial $P$ of degree at most $k$ satisfying
\begin{equation}\label{eq:interpolation problem}
T(x_i) = P(x_i) \quad (i=1,\hdots,k+1).
\end{equation}
This process replaces the nonlinear eigenvalue problem \eqref{eq:NEP} by a \emph{polynomial eigenvalue problem} (PEP)
\begin{equation}\label{eq:PEP}
P(\lambda)x = 0 \quad \quad \mbox{and} \quad \quad y^TP(\lambda)=0.
\end{equation}

If the interpolation error $\max_{\lambda\in\Omega}\|P(\lambda)-T(\lambda)\|_2$ is small, one expects the eigenvalues of $P(\lambda)$ in $\Omega$ and their corresponding eigenvectors to be reliable approximations to the
eigenvalues and eigenvectors of $T(\lambda)$ in a backward error sense \cite{nonlinear}.

One of the most popular techniques for solving polynomial eigenvalue problems is \emph{linearization} \cite{Lancaster}. 
A linearization of a matrix polynomial $P(\lambda)$ replaces \eqref{eq:PEP} with a (larger) \emph{generalized eigenvalue problem}
\begin{equation}\label{eq:GEP}
\lambda Bv = Av \quad \quad \mbox{and} \quad \quad \lambda w^TB=w^TA
\end{equation}
with the same eigenvalues (and multiplicities) as the original PEP.
The linearized eigenvalue problem \eqref{eq:GEP} can be solved by using the QZ algorithm  (for small/medium sizes) or a Krylov method (for larger sizes) \cite{CORK}.

 It is well-known that the linearization transformation is not unique \cite{Amiraslani, Fiedler, MMMM}. Common choices are the Frobenius companion linearizations \cite{Fiedler}, which are based on an expansion of $P(\lambda)$ in the monomial basis
\begin{equation}\label{eq:matrix poly - monomial}
P(\lambda) = \sum_{i=0}^k P_i\,\lambda^i, \quad P_0,\hdots,P_k\in\mathbb{C}^{n\times n}.
\end{equation}
Since polynomial interpolation in the monomial basis can be potentially unstable --due to the ill conditioning of Vandermonde matrices-- we will consider instead matrix polynomials of the form
\begin{equation}\label{eq:matrix poly - other bases}
P(\lambda) = \sum_{i=0}^k P_i \, n_i(\lambda), \quad P_0,\hdots,P_k\in\mathbb{C}^{n\times n},
\end{equation}
where $\{n_i(\lambda)\}_{i=0}^k$ denotes either the Newton, Lagrange or Chebyshev polynomial bases, since these bases are the most common choices for dealing with polynomial interpolants in numerical practice \cite{Chebyshev2, chebfun, Higham}.

In the literature, linearizations of a matrix polynomial expressed in either of these bases can be found in \cite{Amiraslani, chebyshev pep, lawrence-perez, Newton, Robol,Lagrange}.  Among these linearizations, those used most often in applications  can be considered 
 ``equivalent'' to the Frobenius linearizations in the monomial case.  They are called \emph{Colleague linearizations}. 
Our ultimate goal  in a forthcoming paper is to compare the numerical performance (in terms of conditioning and backward errors  \cite{Tisseur1,Tisseur2,Tisseur3}) of the linearizations of a matrix polynomial expressed in the three bases: Chebyshev, Newton, and Lagrange in the following sense. First, we would like to determine if the Colleague linearizations used in practice are the ``best'' linearizations for a given basis. In order to do this analysis, we need  a whole family of linearizations to choose from and compare with. Secondly, once we have chosen the best linearization for each basis, we want to compare the performance of these linearizations for the three given bases in terms of  the selection of nodes for interpolation. The relative position of the eigenvalues with respect to the interpolation nodes has an important effect on the numerical behavior of these linearizations. 

In order to achieve the ultimate goal mentioned above, in this paper, we present three families of strong linearizations for matrix polynomials expressed in the Chebyshev, Newton, and Lagrange bases, respectively. The main reason to construct these families, despite the fact that some families of linearizations already exist for some bases, such as Chebyshev and Newton, is because these available constructions in the literature are implicit (see, for example \cite{Newton,Cheby-Fiedler}, or \cite{bernstein} for the Bernstein basis) and, thus, not   easy to use for the numerical analysis that we intend to do.  Moreover, we have used a block minimal basis approach (\cite{Linearizations})  for the construction of the linearizations (thus, providing their explicit block-structure) which allows equivalent formulations for the three bases. This makes the numerical analysis and comparison much more straightforward. For completion, we give linearizations for both polynomials that are regular and singular, and also provide recovery formulas for eigenvectors, minimal bases, and minimal indices. The numerical analysis and comparison is postponed to a subsequent paper to limit the length of the paper.

As for the structure of the paper,  after some preliminaries (Sections \ref{sec:prelim first}--\ref{sec:prelim last}), where we introduce the notation used throughout the paper and background knowledge, we present in Section \ref{sec:BMBP} the so-called block minimal basis linearizations. 
This family of linearizations was introduced recently in \cite{Linearizations},
and will allow us to construct in Sections \ref{sec:Newton}, \ref{sec:Lagrange} and \ref{sec:Chebyshev} linearizations for matrix polynomials expressed in the Newton, Lagrange and Chebyshev bases, respectively. 
For each of the considered polynomial bases, we introduce an infinite family of linearizations, and for each of these families, we obtain eigenvector formulas, and show how to recover the eigenvectors, minimal indices and minimal bases of the original matrix polynomial from those of any of its linearizations.
Our results put into a unified framework some results scattered in the linearization literature \cite{Amiraslani,lawrence-perez,Robol}, and fill some important gaps in the literature regarding eigenvector formulas, recovery procedures for eigenvectors and minimal bases and minimal indices, and explicit constructions.

\section{Background and notation}
Although most of the definitions and results in this paper hold over a generic field, we focus on the complex numbers.

\subsection{Block vectors and the block transpose}\label{sec:prelim first}

A \emph{block vector} is a matrix of the form
\[
v = 
\begin{bmatrix}
V_1 & V_2 & \cdots & V_n
\end{bmatrix} 
\qquad \mbox{or} \qquad 
v = 
\begin{bmatrix}
V_1 \\ V_2 \\ \vdots \\ V_n 
\end{bmatrix},
\]
where the entries $V_i$ are (possibly) matrices. 
We sometimes use $v(i)$ to denote the $i$th block entry  of a block vector $v$.
The \emph{block transpose} operation, denoted by $\mathcal{B}$, is the blockwise transposition, i.e.,
\[
\begin{bmatrix}
V_1 & V_2 & \cdots & V_n
\end{bmatrix}^\mathcal{B}
= 
\begin{bmatrix}
V_1 \\ V_2 \\ \vdots \\ V_n 
\end{bmatrix} 
\quad \mbox{and} \quad 
\begin{bmatrix}
V_1 \\ V_2 \\ \vdots \\ V_n 
\end{bmatrix}^\mathcal{B} =
\begin{bmatrix}
V_1 & V_2 & \cdots & V_n
\end{bmatrix}.
\]

Note that, in the first case, we are assuming that all the blocks entries have the same number of columns and, in the second case, we are assuming that all the block entries have the same number of rows.

\begin{remark}
Let $V=\left[\begin{smallmatrix} V_1 \\ \vdots \\ V_n \end{smallmatrix}\right]$ be a block vector with block entries $V_i$ all having the same number of columns.
For lack of space reasons, and with a slight abuse of notation, we sometimes write $V = \begin{bmatrix} V_1 & \cdots & V_n \end{bmatrix}^\mathcal{B}$ even when not all the $V_i$ block entries have the same number of rows.
\end{remark}

\subsection{Matrix polynomials}

Let us consider an $m\times n$ matrix polynomial with complex matrix coefficients of the form
\begin{equation}\label{eq:matrix poly - monomial - mbyn}
P(\lambda) = \sum_{i=0}^k P_i\,\lambda^i, \quad P_0,\hdots,P_k\in\mathbb{C}^{m\times n}.
\end{equation}
If $P_k$ is nonzero, we say that $P(\lambda)$ has \emph{degree} $k$; otherwise, we say that $P(\lambda)$ has grade $k$.
We denote the degree of a matrix polynomial $P(\lambda)$ by $\mathrm{deg}\,P(\lambda)$.
When dealing with interpolation polynomials, the notion of grade is  more natural than the notion of degree, since one cannot guarantee a priori a nonzero leading term. 

A matrix polynomial of size $n\times 1$ is called a \emph{(column) vector polynomial}.

We say that a matrix polynomial $P(\lambda)$ is \emph{regular} if $m=n$ and $\det(P(\lambda))$ is not identically zero.
In other words, a regular matrix polynomial $P(\lambda)$ is an invertible matrix over the field $\mathbb{C}[\lambda]$ of rational functions with complex coefficients.
We say that $P(\lambda)$ is \emph{singular} if either $m\neq n$ or $\det(P(\lambda))\equiv 0$.

We say that the matrix polynomial given in \eqref{eq:matrix poly - monomial - mbyn} is expressed in the monomial basis, since $\{1,\lambda,\hdots,\lambda^k\}$ is a basis of the set of polynomials $C_k[\lambda]$ of degree at most $k$ (that is, of grade $k$).
As explained in the introduction, in interpolation problems, it is more convenient to express a matrix polynomial in other polynomial bases.
In the paper, we focus on matrix polynomials expressed either in the Newton, Lagrange or Chebyshev bases. 
We recall these bases next.

\subsection{Polynomial interpolation bases}

\subsubsection{Newton interpolation basis}
For a given set of nodes $\{x_1,\hdots,x_{k+1}\}\in\mathbb{C}$, the Newton polynomial $n_i(\lambda)$ is defined as
\begin{equation}\label{eq:Newton polynomial}
n_i(\lambda)=\prod_{j=1}^i(\lambda-x_j)\quad (i=1,\hdots,k),
\end{equation}
and $n_0(\lambda)=1$.
We notice that the Newton polynomials satisfy the following recurrence relation
\begin{equation}\label{eq:recurrence Newton}
n_i(\lambda) = (\lambda-x_i)n_{i-1}(\lambda) \quad (i=1,\hdots,k).
\end{equation}

The interpolation matrix polynomial, i.e., the unique grade-$k$ matrix polynomial $P(\lambda)$ satisfying \eqref{eq:interpolation problem}, can be written as
\begin{equation}\label{eq:matrix poly - Newton}
P(\lambda) = \sum_{i=0}^k P_i \, n_i(\lambda)
\end{equation}
where the matrix coefficients $P_i\in\mathbb{C}^{n\times n}$ can be found, for example, by using the method of divided differences.
Setting $y_i := T(x_i)$ ($i=1,\hdots,k+1$), the divided differences are defined as
\[
[y_i]:= y_i, \quad [y_i, y_{i+1}, \ldots, y_{i+j}] := \frac{[y_{i+1}, \ldots, y_{i+j}] - [y_i, y_{i+1}, \ldots, y_{i+j-1}] }{x_{i+j} - x_i}.
\]
Then, $P_i= [y_{1}, \ldots, y_{i+1}]$, for $i=0,1,\hdots,k$. 

\subsubsection{Lagrange interpolation basis}
For a given set of nodes  $\{x_1, x_2, \dots, x_{k+1}\}\subset \mathbb{C}$, the \emph{Lagrange polynomial} $\ell_i(\lambda)$ is defined as 
\begin{equation}\label{eq:def_lagrange_poly}
\ell_i(\lambda):= 
\frac{\prod\limits_{j=1, j\neq i}^{k+1}(\lambda-x_j)}{\prod\limits_{j=1, j\neq i}^{k+1}(x_i-x_j)},\quad (i=1,\hdots,k+1).
\end{equation}
The Lagrange polynomial $\ell_i(\lambda)$ has the property
\[
\ell_i(x_j) = \left\{
\begin{array}{ll}
1 & \mbox{ if }j=i, \mbox{ and}\\
0 & \mbox{ otherwise}
\end{array} 
\right.  \quad (i,j=1,\hdots,k+1).
\]
Hence, the unique matrix polynomial $P(\lambda)$ satisfying \eqref{eq:interpolation problem} can be written in terms of Lagrange polynomials as 
\begin{equation}\label{eq:matrix poly in Lagrange basis}
P(\lambda) = \sum_{i=1}^{k+1} P_i\,\ell_i(\lambda),
\end{equation}
where $P_i = T(x_i)$ ($i=1,\hdots,k+1$).

For our purposes, it will be more convenient to express the Lagrange polynomials in the equivalent modified way
\begin{equation}\label{eq:def_modified_lagrange_poly}
\ell_i(\lambda)=\ell(\lambda)\frac{\omega_i}{\lambda-x_i} \quad (i=1,\hdots,k+1),
\end{equation}
where
\begin{equation}\label{weights}
\ell(\lambda)=\prod_{i=1}^{k+1} (\lambda-x_i)\quad \textrm{ and }\quad \omega_i=\frac{1}{\prod\limits_{j\neq i}(x_i-x_j)}\quad (i=1,\hdots,k+1).
\end{equation}
The quantities $\omega_i$ are known as the barycentric weights. 
Using \eqref{eq:def_modified_lagrange_poly}, the matrix polynomial $P(\lambda)$ in \eqref{eq:matrix poly in Lagrange basis} takes the form
\begin{equation}\label{eq:matrix poly in barycentric form}
P(\lambda)=\ell(\lambda)\sum\limits_{i=1}^{k+1} P_i  \frac{\omega_i}{\lambda-x_i},
\end{equation}
which is known as the first barycentric form of \eqref{eq:matrix poly in Lagrange basis}. 
\subsection{The Chebyshev bases of the first and second kind}

The \emph{Chebyshev polynomials  of the first kind} $\{T_n(x): n\in 0\cup \mathbb{N}\}$ are obtained from the recurrence relation
\begin{equation}\label{eq:recurrence Cheby}
T_n(x)= 2x T_{n-1}(x) - T_{n-2}(x),
\end{equation}
where $T_0(x)=1$ and $T_i(x) = x$.
The \emph{Chebyshev polynomials of the second kind} $\{U_n(x): n\in 0\cup \mathbb{N}\}$  are obtained from the same recurrence relation \eqref{eq:recurrence Cheby} with initial conditions $U_0(x)=1$ and $U_1(x)= 2x$.

Chebyshev polynomials can be used to interpolate nonlinear matrix-valued functions $T:[-1,1]\rightarrow \mathbb{C}^{n\times n}$.
Two types of nodes are usually considered: (1) Chebyshev nodes of the first kind
\[
x_i = \cos \left(\frac{2i-1}{k+1}\frac{\pi}{2} \right), \quad i\in \{1, 2, \ldots, k+1\},
\]
and (2) Chebyshev nodes of the second kind
\[
x_i= \cos \left( \frac{i-1}{k}\pi\right),\quad i\in \{1, 2, \ldots k+1\}.
\]
In both cases, the unique grade-$k$ matrix polynomial $P(\lambda)$ satisfying \eqref{eq:interpolation problem} can be written in the form
\begin{equation}\label{eq:matrix poly Chebyshev}
P(\lambda) = \sum_{i=0}^k P_i \,T_i(\lambda),
\end{equation}
where the matrix coefficients $P_i$ ($i=0,\hdots,k$) can be efficiently computed by a sequence of inverse discrete cosine 
transforms of type III or type I, respectively. 
Details can be found in \cite{Chebyshev-Pi}.

\begin{remark}
Although  the Chebyshev polynomials are usually considered to be defined in the real line, there is a generalization of these polynomials in the complex plane: Given a compact set $K\in \mathbb{C}$, the $n$th Chebyshev polynomial associated with $K$ is defined to be the (unique) monic polynomial which minimizes the supremum norm on $K$ among all monic polynomials of the same degree. 
However, as far as we know, there is not a formula to compute these polynomials in an arbitrary set $K$, which is a drawback compared to Newton and Lagrange. 
Thus, in Section \ref{sec:Chebyshev}, we assume the ordinary Chebyshev polynomials defined in the real line.  
\end{remark}

The following lemma will be used in future sections.
\begin{lemma}\cite{lawrence-perez}\label{cheb-ident} The Chebyshev polynomials satisfy the following identities:
\begin{align*}
T_{r+\ell}(\lambda) & = U_r(\lambda)T_{\ell}(\lambda) - U_{r-1}(\lambda)T_{\ell-1}(\lambda) \quad (\ell \neq 0),\\
T_{r+\ell+1}(\lambda) & = 2\lambda U_r(\lambda)T_{\ell}(\lambda)- U_r(\lambda)T_{\ell-1}(\lambda) - U_{r-1}(\lambda)T_{\ell}(\lambda) \quad (\ell \neq 0),\\
U_{r+\ell}(\lambda) & = U_r(\lambda)U_{\ell}(\lambda) - U_{r-1}(\lambda)U_{\ell-1}(\lambda),\\
U_{r+\ell+1}(x) & = 2\lambda U_r(\lambda)U_{\ell}(\lambda)- U_r(\lambda)U_{\ell-1}(\lambda) - U_{r-1}(\lambda)U_{\ell}(\lambda).
\end{align*}
\end{lemma}

\subsection{Eigenvalues and eigenvectors of regular matrix polynomials}

Let $P(\lambda)$ be a regular matrix polynomial of grade $k$ as in \eqref{eq:matrix poly - monomial - mbyn}.
We say that $\lambda_0\in\mathbb{C}$ is a \emph{finite eigenvalue} of $P(\lambda)$ if $P(\lambda_0) x =0$ for some nonzero vector $x$.
 The vector $x$ is called a \emph{right eigenvector} of $P(\lambda)$ associated with $\lambda_0$.
  A vector $y$ is said to be a \emph{left eigenvector} of $P(\lambda)$ associated with $\lambda_0$ if $y^T P(\lambda_0) =0$, where $y^T$ denotes the transpose of $y$. 
  We say that $P(\lambda)$ has an \emph{eigenvalue at infinity} if zero is an eigenvalue of the \emph{$k$-reversal} $\textrm{rev}_kP(\lambda)$ of $P(\lambda)$,  where
\begin{equation}\label{revP}
\textrm{rev}_k P(\lambda) = \lambda^k P\left (1/\lambda \right ).
\end{equation}
In this case, a right (resp. left) eigenvector of $P(\lambda)$ associated with an infinite eigenvalue is a right (resp. left) eigenvector of $\textrm{rev}_k P(\lambda)$ associated with 0.

Two matrix polynomials $P(\lambda)$ and $Q(\lambda)$ of the same size  are said to be \emph{strictly equivalent} if there are invertible matrices $U$ and $V$ such that $Q(\lambda)=UP(\lambda)V$.
We recall that two strictly equivalent matrix polynomials have the same finite and inifinite eigenvalues with the same algebraic, partial and geometric multiplicities.

In future sections, we will consider eigenvalues at infinity of matrix polynomials expressed in polynomial bases other than the monomial.
 The following lemma provides  the reversal of such a polynomial. 
 We omit the proof since it follows immediately from the definition of reversal.

\begin{lemma}\label{revk}
Let $P(\lambda)= \sum_{i=0}^k P_i\, \phi_i(\lambda)$ be a matrix polynomial of grade $k$ expressed in the polynomial basis $\{\phi_0, \phi_1, \ldots, \phi_k\}$. Then,
\[
\mathrm{rev}_k\,P(\lambda)= \sum_{i=0}^k P_i \, \mathrm{rev}_k\,\phi_i(\lambda).
\]
In particular, if $\phi_i(\lambda) =\prod_{j=0}^s  (\lambda - a_j)$, where $s \leq k$, then
\[
\mathrm{rev}_k \,\phi_i (\lambda) = \lambda^{k-s} \prod_{j=0}^s (1- a_j\lambda). 
\]

\end{lemma}

\subsection{Singular matrix polynomials and dual minimal bases}\label{sec:prelim last}

If an $m\times n$ matrix polynomial $P(\lambda)$ is singular, then it has non-trivial left and/or right rational null spaces:
\begin{align*}
&\mathcal{N}_{\ell}(P):=\{y(\lambda) \in \mathbb{C}(\lambda)^{ m \times 1} : y(\lambda)^T P(\lambda)=0\}, \quad \mbox{and}\\
&\mathcal{N}_{r}(P):= \{ x(\lambda)\in \mathbb{C}(\lambda)^{n \times 1}:  P(\lambda)x(\lambda) =0\}.
\end{align*}
Each of these vector spaces contains a  basis consisting of vector polynomials \cite{Forney}. 
We call a basis consisting of vector polynomials a \emph{polynomial basis}. 
The \emph{order} of a polynomial basis is the sum of the degrees of its vectors. 
Among all the  polynomial bases we consider those with least order. 

\begin{definition}[\rm Minimal basis]
Let $\mathcal{V}$ be a rational subspace of $\mathbb{C}(\lambda)^{n\times 1}$. A \emph{minimal basis} of $\mathcal{V}$ is any polynomial basis of $\mathcal{V}$ with least order among all polynomial bases. 
\end{definition}

Minimal bases for a rational subspace $\mathcal{V}$ are not unique, but the ordered list of the degrees of the vector polynomials in each of them is the same.
These degrees are called \emph{the minimal indices} of $\mathcal{V}$ \cite{Forney}.

\begin{definition}[\rm Minimal indices of singular matrix polynomials]
Let $P(\lambda)$ be an $m\times n$ singular matrix polynomial and let  $\{y_1(\lambda)^T, \ldots, y_q(\lambda)^T\}$ and $\{x_1(\lambda), \ldots, x_p(\lambda)\}$ be minimal bases of $\mathcal{N}_{\ell} (P)$ and $\mathcal{N}_r(P)$, respectively, ordered so that $\deg(y_1(\lambda)) \leq \cdots \leq \deg(y_q(\lambda))$ and $\deg(x_1(\lambda)) \leq \cdots \leq \deg(x_p(\lambda))$. Let $\mu_j = \deg(y_j(\lambda))$ for $j=1,2,\ldots, q$, and $\epsilon_j=\deg(x_j(\lambda))$, for $j=1, 2, \ldots, p$. Then, $\mu_1 \leq \ldots \leq \mu_q$ and $\epsilon_1 \leq \ldots \leq \epsilon_p$ are, respectively, the \emph{left and right minimal indices} of $P(\lambda)$. 
\end{definition}

Theorem \ref{minimal-basis} provides a useful characterization of minimal bases.
To state this result, we need the following definition from \cite{row-reduced}.

\begin{definition}
Let $P(\lambda)\in\mathbb{C}[\lambda]^{m\times n}$ be a matrix polynomial with row degrees $d_1,d_2,\dots,d_m$. The {\rm highest row degree coefficient matrix} of $P(\lambda)$, denoted by $P_h$, is the $m\times n$ constant matrix whose $j$th row is the coefficient of $\lambda^{d_j}$ in the $j$th row of $P(\lambda)$, for $j=1,2,\dots,m$. The matrix polynomial $P(\lambda)$ is called {\rm row reduced} if $P_h$ has full row rank.
\end{definition}

\begin{theorem}\label{minimal-basis}{\rm \cite[Theorem 2.14]{row-reduced}}
The rows of a matrix polynomial $P(\lambda)$ are a \normalfont{minimal basis} of the rational subspace they span if and only if  $P(\lambda_0)$ has full row rank for all $\lambda_0\in \mathbb{C}$ and $P(\lambda_0)$ is row reduced. A matrix polynomial is called \emph{minimal basis} if its rows form a minimal basis of the rational subspace they span.
\end{theorem}

The linearizations for matrix polynomials that we introduce in the following section use the notion of dual minimal bases \cite{Forney}.
\begin{definition}[Dual minimal bases]
Two matrix polynomials $K(\lambda)\in\mathbb{F}[\lambda]^{m_1\times n}$ and $D(\lambda)\in\mathbb{F}[\lambda]^{m_2\times n}$ are said to be {\rm dual minimal bases} if $K(\lambda)$ and $D(\lambda)$ are both minimal bases, $m_1 + m_2 = n$, and $K(\lambda)D(\lambda)^T = 0$.
\end{definition}

\subsection{Strong linearizations of matrix polynomials, and block minimal basis pencils}
\label{sec:BMBP}

A matrix pencil $L(\lambda)$ is said to be a \emph{linearization} of a matrix polynomial $P(\lambda)$ as in \eqref{eq:interpolation problem} if there exist a positive integer $s$ and two unimodular matrices (i.e., matrix polynomials whose determinant is a nonzero constant) $U(\lambda)$ and $V(\lambda)$ such that
\[
U(\lambda)L(\lambda)V(\lambda)  = 
\begin{bmatrix}
I_s & 0 \\
0 & P(\lambda)
\end{bmatrix}
\]
A linearization $L(\lambda)$ of a grade-$k$ matrix polynomial $P(\lambda)$ is \emph{strong} if $\mathrm{rev}_1 \, L(\lambda)$ is a linearization of $\mathrm{rev}_k\,P(\lambda)$ \cite{Lancaster}. 

\begin{remark}
A strong linearization of a matrix polynomial $P(\lambda)$ preserves the finite and infinite eigenvalues of $P(\lambda)$ and their multiplicities, and the dimension of the right and left nullspaces.
\end{remark}
\begin{remark}
Any matrix pencil strictly equivalent to a strong linearization of a matrix polynomial $P(\lambda)$ is also a strong linearization of $P(\lambda)$.
\end{remark}

One of our main objective in this paper is to find strong linearizations for matrix polynomials of the form
\begin{equation}\label{eq: matrix poly - other bases - mbyn}
P(\lambda) = \sum_{i=0}^k P_i\, \phi_i(\lambda), \quad P_0,\hdots,P_k\in\mathbb{C}^{m\times n},
\end{equation}
where $\{\phi_i\}$ denotes either the Newton, Lagrange or Chebyshev bases, that can be easily constructed from the coefficients $P_i$ and the nodes. 
We will find such linearizations in the family of so-called block minimal basis pencils \cite{Linearizations}.

\begin{definition}[Block minimal basis pencils]
A matrix pencil 
\begin{equation}\label{eq:bmbp def}
L(\lambda)=
\left[ \begin{array}{c|c}
M(\lambda) & K_2(\lambda)^T\\
\hline
K_1(\lambda) & 0
\end{array}\right]
\end{equation}
is called a \emph{block minimal basis pencil} if  $K_1(\lambda)$ and $K_2(\lambda)$ are both minimal bases.
 If, in addition, the row degrees of $K_1(\lambda)$ are all equal to 1, the row degrees of $K_2(\lambda)$ are all equal to 1, the row degrees of a minimal basis dual to $K_1(\lambda)$ are all equal and the row degrees of a minimal basis dual to $K_2(\lambda)$ are equal, then $L(\lambda)$ is a \emph{strong block minimal basis pencil}. 
 The submatrix $M(\lambda)$ is called the \emph{body of $L(\lambda)$}. 
\end{definition}

Theorems \ref{thm:key1} and \ref{thm:key2} are two key results on strong block minimal basis pencils.
Theorem \ref{thm:key1} says that every strong block minimal basis pencil is always a strong linearization of a certain matrix polynomial. 
\begin{theorem}\label{thm:key1}{\rm\cite{Linearizations}}
Let $K_1(\lambda)$ and $D_1(\lambda)$, and $K_2(\lambda)$ and $D_2(\lambda)$ be two pairs of dual minimal bases, let $L(\lambda)$ be a strong block minimal basis pencil as in \eqref{eq:bmbp def}, and let
\begin{equation}\label{eq:Q}
Q(\lambda):= D_2(\lambda)M(\lambda)D_1(\lambda)^T.
\end{equation}
Then:
\begin{itemize}
\item[\rm(a)] $L(\lambda)$ is a linearization of $Q(\lambda)$.
\item[\rm(b)] If $L(\lambda)$ is a strong block minimal basis pencil, then $L(\lambda)$ is a strong linearization of $Q(\lambda)$, considered as a polynomial with grade $1+\deg(D_1(\lambda))+\deg(D_2(\lambda))$.
\end{itemize}
\end{theorem}

Theorem \ref{thm:key2} says essentially two things: 1)
 given a matrix polynomial $P(\lambda)$, it says that we can always find a pencil $M(\lambda)$ such that the strong block minimal basis pencil \eqref{eq:bmbp def} is a strong linearization of $P(\lambda)$; 
2) it provides a characterization of all the pencils $M(\lambda)$ that make the block minimal basis pencil \eqref{eq:bmbp def} a strong linearization of the given polynomial $P(\lambda)$.

\begin{theorem}\label{thm:key2}{\rm \cite{ell-ifications}}
Let $P(\lambda)$ be an $m\times n$ matrix polynomial, let $K_1(\lambda)$ and $D_1(\lambda)$, and $K_2(\lambda)$ and $D_2(\lambda)$ be two pairs of dual minimal bases such that $D_1(\lambda)$ has $n$ rows, $D_2(\lambda)$ has $m$ rows, and $\deg(P(\lambda))\leq 1 +  \deg(D_1(\lambda))+\deg(D_2(\lambda))$, and let $L(\lambda)$ be a strong block minimal basis pencil as in \eqref{eq:bmbp def}.
Then:
\begin{itemize}
\item[\rm(a)] The linear equation
\begin{equation}\label{eq:M given P}
P(\lambda) = D_2(\lambda)M(\lambda)D_1(\lambda)^T
\end{equation}
is solvable for the matrix pencil $M(\lambda)$. 
\item[\rm(b)] If $M_0(\lambda)$ is a solution of \eqref{eq:M given P}, then any other solution is of the form
\[
M(\lambda) = M_0(\lambda)+AK_1(\lambda)+K_2(\lambda)^T B,
\]
for some constant matrices $A$ and $B$.
\end{itemize}
\end{theorem}

\begin{remark}
For the linearizations introduced in Sections \ref{sec:Newton}, \ref{sec:Lagrange} and \ref{sec:Chebyshev}, we will be able to construct a matrix pencil $M(\lambda)$ satisfying \eqref{eq:M given P} directly from the matrix coefficients of the matrix polynomial $P(\lambda)$.
\end{remark}

Theorem \ref{thm:factorizations} will allow us to prove that the linearizations we introduce in this work are more than strong linearizations, since we will be able to recover minimal indices, minimal bases and left and right eigenvectors of the original matrix polynomial $P(\lambda)$ from those of its linearizations.
Due to its technicality, we postpone the proof of Theorem \ref{thm:factorizations} to the Appendix.
\begin{theorem}\label{thm:factorizations}
Let $P(\lambda)$ be an $m\times n$ matrix polynomial as in \eqref{eq:matrix poly - monomial - mbyn}, let $K_1(\lambda)$ and $D_1(\lambda)$, and $K_2(\lambda)$ and $D_2(\lambda)$ be two pairs of dual minimal bases, and let $L(\lambda)$ be a strong block minimal basis pencil as in \eqref{eq:bmbp def} such that
\[
P(\lambda) = D_2(\lambda)M(\lambda)D_1(\lambda)^T.
\]
Suppose right- and left-sided factorizations of the form
\[
L(\lambda)
\begin{bmatrix}
D_1(\lambda)^T\\ X(\lambda)
\end{bmatrix}=v\otimes P(\lambda) \quad \mbox{and} \quad 
\begin{bmatrix}
D_2(\lambda) \; Y(\lambda)^T
\end{bmatrix}=w^T\otimes P(\lambda),
\]
hold for some matrix polynomials $X(\lambda)$ and $Y(\lambda)$, and for some nonzero vectors $v,w\in\mathbb{C}^k$.

Assume $m=n$ and  $P(\lambda)$ is regular. 
If $\lambda_0$ is a finite eigenvalue of $P(\lambda)$ with geometric multiplicity $g$, then
\begin{itemize}
\item[\rm(a)] $\{x_1,\hdots, x_g\}$ is a basis for $\mathcal{N}_r(P(\lambda_0))$ if and only if $\{v_1,\hdots,v_g\}$  is a basis for $\mathcal{N}_r(L(\lambda_0))$, where $v_i=\left[\begin{smallmatrix} D_1(\lambda_0)^T \\ X(\lambda_0) \end{smallmatrix} \right]x_i$, for $i=1,\hdots,g$.
\item[\rm(b)] $\{y_1,\hdots,y_g\}$ is a basis for $\mathcal{N}_\ell(P(\lambda_0))$ if and only if $\{w_1,\hdots,v_g\}$  is a basis for $\mathcal{N}_r(L(\lambda_0))$, where $w_i=\left[\begin{smallmatrix} D_2(\lambda_0)^T \\ Y(\lambda_0) \end{smallmatrix} \right]y_i$, for $i=1,\hdots,g$.
\end{itemize} 
Assume $P(\lambda)$ is singular.
If $\dim\,\mathcal{N}_r(P(\lambda))=p$ and $\dim\,\mathcal{N}_\ell(P(\lambda))=q$, then
\begin{itemize}
\item[\rm(c)] $\{x_1(\lambda),\hdots,x_p(\lambda)\}$ is a minimal basis for $\mathcal{N}_r(P(\lambda))$ if and only if $\{v_1(\lambda),\hdots,v_p(\lambda)\}$  is a minimal basis for $\mathcal{N}_r(L(\lambda))$, where $v_i=\left[\begin{smallmatrix} D_1(\lambda)^T \\ X(\lambda) \end{smallmatrix} \right]x_i(\lambda)$, for $i=1,\hdots,p$.
\item[\rm(d)] $\{y_1(\lambda),\hdots,y_q(\lambda)\}$ is a basis for $\mathcal{N}_\ell(P(\lambda))$ if and only if $\{w_1(\lambda),\hdots,w_q(\lambda)\}$  is a basis for $\mathcal{N}_\ell(L(\lambda))$, where $w_i(\lambda)=\left[\begin{smallmatrix} D_2(\lambda)^T \\ Y(\lambda) \end{smallmatrix} \right]y_i(\lambda)$, for $i=1,\hdots, q$.
\end{itemize}
Moreover, if $0\leq\epsilon_1\leq \epsilon_2\leq \cdots \leq \epsilon_p$ are the right minimal indices of $P(\lambda)$, and $0\leq\mu_1\leq \mu_2\leq \cdots\leq \mu_q$ are the left minimal indices of $P(\lambda)$, then
\begin{itemize}
\item[\rm(e)] $\epsilon_1+\deg(D_1(\lambda))\leq \epsilon_2+\deg(D_1(\lambda))\leq\cdots\leq \epsilon_q + \deg(D_1(\lambda))$ are the right minimal indices of $L(\lambda)$, and 
\item[\rm(f)] $\mu_1+\deg(D_2(\lambda))\leq \mu_2+\deg(D_2(\lambda))\leq\cdots\leq \mu_q + \deg(D_2(\lambda))$ are the left minimal indices of $L(\lambda)$. 
\end{itemize}
\end{theorem}

\section{Strong linearizations for matrix polynomials in the Newton basis}\label{sec:Newton}
Let $\{x_1,\hdots,x_k\}$ be a set of $k$ distinct nodes, and let $P(\lambda)=\sum_{i=0}^k P_i\,n_i(\lambda)$ be an $m\times n$ matrix polynomial expressed in the Newton basis associated with this set of nodes.

Associated with the set of nodes $\{x_1,x_2,\hdots,x_k\}$ we introduce the following polynomials
\begin{equation}\label{eq:gamma}
\gamma_j(\lambda):=\lambda-x_j, \quad \quad (j=1,2,\hdots,k),
\end{equation}
and
\begin{equation}\label{eq:nij}
n_i^j(\lambda):=\left\{ \begin{array}{ll}
\prod_{\ell=i}^j\gamma_\ell(\lambda) & \mbox{ if }j\geq i,\\
1 & \mbox{ if }j<i,
\end{array}\right. 
\quad (i,j=1,2,\hdots,k).
\end{equation}
Notice that $n_1^j(\lambda)$ is just the $j$th Newton polynomial $n_j(\lambda)$, for $j=1,2,\hdots,k$.

Let $0\leq \mu\leq k-1$ be an integer and let $n$ and $m$ be  positive integers. 
We define the matrix pencils
\begin{align}
\label{eq:K1 Newton}
&K_1^N(\lambda):= 
\begin{bmatrix}
	-I_n & \gamma_{k-1}(\lambda)I_n \\
	& -I_n & \gamma_{k-2}(\lambda)I_n \\
	& & \ddots & \ddots \\
	& & & -I_n & \gamma_{\mu+1}(\lambda)I_n
\end{bmatrix} \quad \mbox{and}\\
\label{eq:K2 Newton}
&K_2^N(\lambda):= 
\begin{bmatrix}
	-I_m & \gamma_{\mu}(\lambda)I_m \\
	& -I_m & \gamma_{\mu-1}(\lambda)I_m \\
	& & \ddots & \ddots \\
	& & & -I_m & \gamma_{1}(\lambda)I_m
\end{bmatrix},
\end{align}
where the polynomials $\gamma_j(\lambda)$ are defined in \eqref{eq:gamma}, and where the empty block-entries are assumed to be zero blocks. 
 We note that, if $\mu=0$ (resp. $\mu=k-1$), the matrix $K_2^N(\lambda)$ (resp.  $K_1^N(\lambda)$) is an empty matrix.

\begin{lemma}\label{lemma:min bases Newton}
Let $\{x_1,\hdots,x_k\}$ be a set of distinct nodes, and let $0\leq \mu\leq k-1$ be an integer.
The matrix pencils $K_1^N(\lambda)$ and $K_2^N(\lambda)$ defined, respectively, in \eqref{eq:K1 Newton} and \eqref{eq:K2 Newton} are minimal bases when they are not empty.
Moreover, in this case, 
\begin{equation}\label{eq:D Newton}
D_1^N(\lambda)^T:=
\begin{bmatrix}
	n_{\mu+1}^{k-1}(\lambda)I_n\\[6pt]
	n_{\mu+1}^{k-2}(\lambda)I_n\\
	\vdots \\
	n_{\mu+1}^{\mu+1}(\lambda)I_n\\
	I_n
\end{bmatrix} \quad \mbox{and} \quad
D_2^N(\lambda)^T:=
\begin{bmatrix}
n_\mu(\lambda)I_m\\[6pt]
n_{\mu-1}(\lambda)I_m\\
\vdots\\
n_1(\lambda)I_m\\
I_m
\end{bmatrix},
\end{equation}
where the $n_i^j(\lambda)$ polynomials are defined in \eqref{eq:nij}, are dual minimal bases of $K_1^N(\lambda)$ and $K_2^N(\lambda)$, respectively.
\end{lemma}
\begin{proof}
The minimality of $K_1^N(\lambda)$, $K_2^N(\lambda)$, $D_1^N(\lambda)$ and $D_2^N(\lambda)$ follows immediately from the characterization of minimal bases in Theorem \ref{minimal-basis}.
The duality of the pairs $(K_1^N(\lambda),D_1^N(\lambda))$ and $(K_2^N(\lambda),D_2^N(\lambda))$ can be established by direct matrix multiplication.
\end{proof}

We now consider strong block minimal basis pencils of the form
\begin{equation}\label{eq:Newton pencil}
L(\lambda) = \begin{bmatrix}
M(\lambda) & K_2^N(\lambda)^T \\
K_1^N(\lambda) & 0 
\end{bmatrix}.
\end{equation}
We will refer to \eqref{eq:Newton pencil} as a \emph{Newton pencil}.
In Theorem \ref{thm:Newton-colleague}, we show how to choose the body of a Newton pencil $L(\lambda)$ as in \eqref{eq:Newton pencil} so  that $L(\lambda)$ is a strong linearization  of a prescribed matrix polynomial.
\begin{theorem}\label{thm:Newton-colleague}
Let $P(\lambda)=\sum_{i=0}^k P_i\,n_i(\lambda)$ be an $m\times n$ matrix polynomial expressed in the Newton basis associated with the nodes $\{x_1,\hdots,x_k\}$. 
Let $0\leq\mu\leq k-1$ be an integer, and let
\begin{equation}\label{eq:M Newton}
M_\mu^N(\lambda):=
\left[
\begin{array}{c|c}
	\begin{matrix}
	\gamma_k(\lambda)P_k+P_{k-1} & P_{k-2} & \cdots & P_{\mu+1} 
	\end{matrix} & P_\mu \\ \hline
	0 & \begin{matrix}
	P_{\mu-1} \\ P_{\mu-2} \\ P_2 \\ P_1 \\ P_0
	\end{matrix}
\end{array}
\right].
\end{equation}
Then, the Newton pencil
\begin{equation}\label{eq:Newton colleague}
N_P^\mu(\lambda):= \begin{bmatrix}
M_\mu^N(\lambda) & K_2^N(\lambda)^T \\
K_1^N(\lambda) & 0
\end{bmatrix}
\end{equation}
is a strong linearization of $P(\lambda)$.
We will refer to \eqref{eq:Newton colleague} as the \emph{colleague Newton pencil} of $P(\lambda)$ associated with $\mu$. 
\end{theorem}
\begin{proof}
By direct matrix multiplication, we have $D_2^N(\lambda)M_\mu^N(\lambda)D_1^N(\lambda)^T=P(\lambda)$.
Hence, by Theorem \ref{thm:key1} together with Lemma \ref{lemma:min bases Newton}, the colleague Newton pencil $N_P^\mu(\lambda)$ is a strong linearization of $P(\lambda)$.
\end{proof}

Using Theorem \ref{thm:key2}, we can now construct an infinite family of Newton pencils that are strong linearizations of a prescribed $m\times n$ matrix polynomial $P(\lambda)$ expressed in the Newton basis.
\begin{theorem}\label{thm:Newton linearizations}
Let $P(\lambda)=\sum_{i=0}^k P_i\,n_i(\lambda)$ be an $m\times n$ matrix polynomial expressed in the Newton basis associated with the nodes $\{x_1,\hdots,x_k\}$ and let $0\leq\mu\leq k-1$ be an integer.
Let $M_\mu^N(\lambda)$ be defined as in \eqref{eq:M Newton}, and let $A$ and $B$ be two arbitrary matrices of size $(\mu+1)m\times (k-\mu-1)n$ and $\mu m \times (k-\mu) n$, respectively.
Then, the Newton pencil
\begin{equation}\label{eq:Newton linearization}
\mathcal{N}(\lambda) = 
\begin{bmatrix}
M_\mu^N(\lambda)+AK_1^N(\lambda)+K_2^N(\lambda)^TB & K_2^N(\lambda)^T \\
K_1^N(\lambda) & 0
\end{bmatrix}.
\end{equation}
is a strong linearization of $P(\lambda)$.
We will refer to \eqref{eq:Newton linearization} as a Newton linearization of the matrix polynomial $P(\lambda)$.
\end{theorem}
\begin{remark}\label{remark:Newton linearization form}
Note that every Newton linearization \eqref{eq:Newton linearization} of a matrix polynomial $P(\lambda)$ can be factored as
\begin{equation*}
\begin{bmatrix}
I_{(\mu+1)m} & A \\
0 & I_{(k-\mu-1)n} \\
\end{bmatrix}
\begin{bmatrix}
M_\mu^N(\lambda) & K_2^N(\lambda)^T \\
K_1^N(\lambda) & 0
\end{bmatrix}
\begin{bmatrix}
I_{(k-\mu)n} & 0 \\
B & I_{\mu m} 
\end{bmatrix}.
\end{equation*}
Hence,  for a fixed  integer $\mu$, all Newton linearizations of the form \eqref{eq:Newton linearization} are strictly equivalent to the colleague Newton pencil \eqref{eq:Newton colleague}.
Notice that, in particular, the matrix $A$ (resp. $B$) can be chosen to contain a single nonzero block-entry, which can be interpreted as an elementary (e.g. Gaussian) block-row (resp. block-column) operation on the matrix pencil \eqref{eq:Newton colleague}.
Using this idea, we produce some examples of Newton linearizations in Example \ref{ex:Newton linearizations}
\end{remark}
\begin{example}\label{ex:Newton linearizations}
Let $P(\lambda)= \sum_{i=0}^5 P_i \, n_i(\lambda)$ be an $m\times n$ matrix polynomial of degree 5 expressed in the Newton basis.
 Let $\mu=2$. Then, the Newton colleague linearization of $P(\lambda)$ associated with $\mu$ is given by
\[
\mathcal{N}_P^2(\lambda)= \left[ \begin{array}{ccc|cc} \gamma_5(\lambda) P_5 +P_4 &  P_3& P_2 & -I_m & 0 \\ 
0 & 0 & P_1 & \gamma_2(\lambda) I_m & -I_m\\ 0 & 0 & P_0 & 0 & \gamma_1(\lambda) I_m\\
\hline
-I_n & \gamma_4(\lambda) I_n & 0 & 0 & 0 \\ 
0 & -I_n & \gamma_3(\lambda) I_n & 0 & 0 
\end{array} \right ].
\]
By Theorem \ref{thm:Newton linearizations}, the following Newton pencils are also strong linearizations of $P(\lambda)$. 
They are obtained from $\mathcal{N}_P^2(\lambda)$ by applying a finite number of elementary block-row or block-column operations.
 Using the notation in Theorem \ref{thm:Newton linearizations}, we specify the matrices $A$ and $B$ used to obtain the body of each particular linearization.
For lack of space, we omit the dependence in $\lambda$ of the $\gamma_i(\lambda)$ polynomials.

The following linearization has been obtained from $\mathcal{N}_P^2(\lambda)$ by adding to the first block-row the fifth block-row multiplied by $P_3$: 
{\small $$\mathcal{N}_1(\lambda)= \left[ \begin{array}{ccc|cc} \gamma_5 P_5 + P_4 &0 & \gamma_3P_3+P_2 & -I_m & 0 \\ 
0 & 0 & P_1 & \gamma_2 I_m & -I_m\\0 & 0 & P_0 & 0 & \gamma_1 I_m\\
\hline
-I_n & \gamma_4 I_n & 0 & 0 & 0 \\ 
0 & -I_n & \gamma_3 I_n & 0 &0 
\end{array} \right ], \quad A = \left [\begin{array}{cc}  0 & P_3 \\ 0 & 0 \\ 0 & 0 \end{array}  \right], \quad B = 0.$$}%
The  following linearization  has been obtained from $\mathcal{N}_1(\lambda)$ by adding to the first block-row the fourth block-row multiplied by $P_4$: 
{\small $$\mathcal{N}_2(\lambda)= \left[ \begin{array}{ccc|cc} \gamma_5 P_5  & \gamma_4 P_4 & \gamma_3 P_3+ P_2 & -I_n &0 \\ 
0&0& P_1 & \gamma_2 I_n & -I_n\\ 0&0& P_0 & 0& \gamma_1 I_n\\
\hline
-I_n & \gamma_4 I_n & 0\ & 0 & 0 \\
0 & -I_n & \gamma_3 I_n  & 0 & 0
\end{array} \right ], \quad A = \left [\begin{array}{cc} P_4 & P_3 \\ 0 & 0 \\ 0 & 0 \end{array}  \right], \quad B = 0.$$}%
The following linearization has been obtained from $\mathcal{N}_2(\lambda)$ by adding to the first block-column the fifth block-column multiplied by $P_1$: 
{\small $$ \left[ \begin{array}{ccc|cc} \gamma_5 P_5  & \gamma_4 P_4 & \gamma_3 P_3+ P_2 & -I_m & 0 \\ 
-P_1 &0 & P_1 & \gamma_2 I_m & -I_m\\
 \gamma_1 P_1 & 0 & P_0 & 0 & \gamma_1 I_m\\
\hline
-I_n & \gamma_4 I_n & 0 & 0 & 0 \\ 
0 & -I_n & \gamma_3 I_n & 0 & 0
\end{array} \right ], \; A = \left [\begin{array}{cc} P_4 & P_3 \\ 0 & 0 \\ 0 & 0 \end{array}  \right], \; B = \left [\begin{array}{ccc} 0 & 0 & 0\\ P_1 & 0  & 0\end{array}  \right].$$}
\end{example}
\begin{remark}
In the literature, a family of strong linearizations of a matrix polynomial $P(\lambda)$ expressed in the Newton basis can be found in \cite{Newton}.
The pencils in this family receive the name of Newton-Fiedler pencils, since they generalize the family of Fiedler pencils \cite{Fiedler}.
As the Newton linearizations, Newton-Fiedler pencils can be easily constructed from the coefficients $P_i$ and the nodes $\{x_1,\hdots,x_k\}$.
However, one of the drawbacks of the family of Newton-Fiedler pencils is that it contains finitely many pencils.
In contrast to this, the family of Newton linearizations contains infinitely many strong linearizations.
Being a larger set, it is more likely to find linearizations with ``good'' numerical and/or structural properties.
Moreover, the Newton-Fiedler pencils are defined implicitly as products of matrices, while the Newton linearizations, being block minimal basis pencils, are given in an explicit way.
\end{remark}

In the following two sections, we will show how to recover the eigenvectors, minimal indices and minimal bases of a matrix polynomial from those of its Newton linearizations.
We will need the following definition.
\begin{definition}[Newton-Horner shifts]\label{def:Newton-Horner}
Given a matrix polynomial $P(\lambda)=\sum_{i=0}^k P_i\,n_i(\lambda)$ expressed in the Newton basis associated with nodes $\{x_1,\hdots,x_k\}$, the $i$th \emph{Newton-Horner shift} of $P(\lambda)$ is given by
\[
P^i(\lambda):= P_k\,n^k_{k+1-i}(\lambda)+P_{k-1}\,n_{k+1-i}^{k-1}(\lambda)+\cdots+P_{k+1-i}\,n_{k+1-i}^{k+1-i}(\lambda)+P_{k-i},
\]
where the $n_i^j(\lambda)$ polynomials are defined in \eqref{eq:nij}.
In particular,  $P^1(\lambda)=P_k\,n^k_k(\lambda)+P_{k-1}$ and $P^k(\lambda)=P(\lambda)$. 
\end{definition}

Newton-Horner shifts satisfy the following recurrence relation
\begin{equation}\label{eq:Newton-Horner recurrence}
P^{i+1}(\lambda)=\gamma_{k-i}(\lambda)\,P^i(\lambda)+P_{k-i-1} \quad (i=1,\hdots,k-1),
\end{equation}
where $\gamma_{k-i}(\lambda)$ is as in \eqref{eq:gamma}.

Theorem \ref{thm:Newton factorizations} gives right- and left-sided factorizations of the Newton colleague pencil \eqref{eq:Newton colleague}.
\begin{theorem}\label{thm:Newton factorizations}
Let $P(\lambda)=\sum_{i=0}^k P_i\,n_i(\lambda)$ be an $m\times n$ matrix polynomial expressed in the Newton basis associated with nodes $\{x_1,\hdots,x_k\}$.
Let $0\leq\mu\leq k-1$ be an integer, let $\mathcal{N}_P^\mu(\lambda)$ be the Newton colleague pencil in \eqref{eq:Newton colleague}, and let $D_1^N(\lambda)$ and $D_2^N(\lambda)$ be the minimal bases in \eqref{eq:D Newton}.

For $0<\mu\leq k-1$, let
\[
H_N^\mu(\lambda)^{\mathcal{B}}:=
\begin{bmatrix}
D_1^N(\lambda) & P^{k-\mu}(\lambda) & \cdots & P^{k-2}(\lambda) & P^{k-1}(\lambda)
\end{bmatrix},
\]
and for $\mu=0$, let \[
H_N^\mu(\lambda)^{\mathcal{B}}:=D_1^N(\lambda) = 
\begin{bmatrix}
n_{k-1}(\lambda)I_n & n_{k-2}(\lambda)I_n & \cdots & n_1(\lambda)I_n & I_n
\end{bmatrix}.
\]
For $0\leq \mu<k-1$, let
\[
G_N^\mu(\lambda):=
\begin{bmatrix}
D_2^N(\lambda) & n_\mu(\lambda)\,P^1(\lambda) & n_\mu(\lambda)\,P^2(\lambda) & \cdots & n_\mu(\lambda)\, P^{k-\mu-1}(\lambda)
\end{bmatrix},
\]
and for $\mu=k-1$, let
\[
G_N^\mu(\lambda):= D_2^N(\lambda)=
\begin{bmatrix}
n_{k-1}(\lambda)I_m & n_{k-2}(\lambda)I_m & \cdots & n_1(\lambda)I_m & I_m 
\end{bmatrix}.
\]
Then, the following right- and left-sided factorizations hold
\[
\mathcal{N}_P^\mu(\lambda) H_N^\mu(\lambda) = e_{\mu+1}\otimes P(\lambda) \quad \mbox{and} \quad 
G_N^\mu(\lambda)\mathcal{N}_P^\mu(\lambda) = e_{k-\mu}^T\otimes P(\lambda),
\]
where the vector $e_i$ denotes the $i$th column of the $k\times k$ identity matrix.
\end{theorem}
\begin{proof}
With the help of the recurrence \eqref{eq:Newton-Horner recurrence} and the fact that $n_{i+1}(\lambda) = \gamma_{i+1}(\lambda)\, n_i(\lambda)$, the results can be directly checked by multiplying $\mathcal{N}_P^\mu(\lambda) H_N^\mu(\lambda)$ and $G_N^\mu(\lambda)\mathcal{N}_P^\mu(\lambda)$
\end{proof}

\subsection{Recovery of eigenvectors from Newton linearizations}

Assume that the matrix polynomial $P(\lambda)=\sum_{i=0}^k P_i\,n_i(\lambda)$ is regular.
In this section, we provide recovery formulas for the (left and right) eigenvectors of $P(\lambda)$ from those of its Newton linearizations.

We start by giving a close formula for the right and left eigenvectors of the Newton colleague pencil \eqref{eq:Newton colleague} associated with its finite eigenvalues.

\begin{theorem}\label{thm:eig-Newton}
Let $P(\lambda)= \sum_{i=0}^k P_i\, n_i(\lambda)$  be an $n \times n$ regular matrix polynomial expressed in the Newton basis associated with nodes $\{x_1,\hdots,x_k\}$.
Let $\lambda_0$ be a finite eigenvalue of $P(\lambda)$.  Let $0 \leq \mu \leq k-1$ be an integer and let $\mathcal{N}_P^{\mu}(\lambda)$ be the Newton colleague pencil in \eqref{eq:Newton colleague}.
Then, $z$ (resp. $\omega$) is a right (resp. left) eigenvector of $\mathcal{N}_P^{\mu}(\lambda)$ associated with $\lambda_0$ if and only if $z= H_N^{\mu}(\lambda_0) x$ (resp. $\omega = G_N^{\mu}(\lambda_0)^T y$), where $x$ (resp. $y$) is a right (resp. left) eigenvector of $P(\lambda)$ associated with $\lambda_0$. 
\end{theorem}
\begin{proof}
If follows immediately from Theorems \ref{thm:factorizations} and \ref{thm:Newton factorizations}.
\end{proof}

The next result provides recovery formulas of eigenvectors associated with finite and infinite eigenvalues of a matrix polynomial from those of its Newton linearizations. 
The eigenvectors of the linearizations are considered block vectors of length $k$ with block-entries of length $n$.
\begin{theorem}[Recovery of eigenvectors from Newton linearizations]\label{recover-eig}
Let $P(\lambda)= \sum_{i=0}^k P_i n_i(\lambda)$ be an $n \times n$ regular matrix polynomial expressed in the Newton basis associated with nodes $\{x_1,\hdots,x_k\}$.
Let $\lambda_0$ be an eigenvalue of $P(\lambda)$.  Let $N(\lambda)$ be a Newton  linearization of $P(\lambda)$ as in \eqref{eq:Newton linearization}.  
Let $z$ and $\omega$ be, respectively, a right and a left eigenvector of $N(\lambda)$ associated with $\lambda_0$. 
\begin{enumerate}
\item Assume $\lambda_0$ is finite. 
Then,
\begin{itemize}
\item $z(k-\mu)$ is a right eigenvector of $P(\lambda)$ associated with $\lambda_0$. 
If, in addition, $\lambda_0 \notin \{x_{\mu+1}, \ldots, x_{k-1}\}$, then  the block-entries  $z(1), z(2), \ldots z(k-\mu)$  are also right eigenvectors of $P(\lambda)$ associated with $\lambda_0$. 
\item $\omega(\mu+1)$ is a left eigenvector of $P(\lambda)$ associated with $\lambda_0$.
If, in addition, $\lambda_0 \notin \{x_{1}, \ldots, x_{\mu}\}$, then  the block-entries  $\omega(1), \omega(2), \ldots \omega(\mu+1)$  are left eigenvectors of $P(\lambda)$ associated with $\lambda_0$. 
\end{itemize}
\item Assume $\lambda_0$ is infinite. Then, 
\begin{itemize}
\item  $z(1)$ is a right eigenvector of $P(\lambda)$ associated with $\lambda_0$. 
\item $\omega(1)$ is a left eigenvector of $P(\lambda)$ associated with $\lambda_0$.
\end{itemize}
\end{enumerate}

\end{theorem}

\begin{proof}
We prove the result for the right eigenvectors. The proof is similar for the left eigenvectors. 

We show first that the theorem holds for the Newton colleague pencil $\mathcal{N}_P^{\mu}(\lambda)$. 

Case I: Assume that $\lambda_0$ is a finite eigenvalue. 
By Theorem \ref{thm:eig-Newton},  $z= H_{N}^{\mu}(\lambda_0) x$ for some eigenvector $x$ of $P(\lambda)$ associated with $\lambda_0$. 
Since the $(k-\mu)$th block-entry of $H_{N}^{\mu}(\lambda_0)$  is the identity matrix, we have that $z(k-\mu)=x$ is a right eigenvector of $P(\lambda)$ with eigenvalue $\lambda_0$.
Further, if $\lambda_0 \notin \{ x_{\mu+1}, \ldots, x_{k-1}\}$, then all the block-entries of $H_{N}^{\mu}(\lambda_0)x$ in positions $1, 2, \ldots, k-\mu-1$ are nonzero multiples of the vector $x$.
Hence, $z(1),\hdots,z(k-\mu)$ are all eigenvectors of $P(\lambda)$ with eigenvalue $\lambda_0$.

Case II: Assume that $\lambda_0$ is an infinite eigenvalue. 
This implies that 0 is an eigenvalue of $\mathrm{rev}_k \, P(\lambda)$ and $\mathrm{rev}_1\, \mathcal{N}_P^{\mu}(\lambda)$. 
 By Lemma \ref{revk}, we have
\[
\textrm{rev}_k \,P(\lambda)=\sum_{i=0}^k P_i \,\textrm{rev}_k\, n_i(\lambda)= \sum_{i=0}^k P_i \,\lambda^{k-i}  \widetilde{n_i}(\lambda),
\]
where $\widetilde{n_i}(\lambda) = \prod_{j=1}^i (1 - x_j \lambda)$.  
Thus, $\textrm{rev}_k \,P(0) = P_k$, which implies that $x$ is a right eigenvector of $P(\lambda)$ with eigenvalue at  infinity if and only if $x$ is a right  eigenvector of $P_k$ with eigenvalue 0. 
Moreover, we have
\[
\mathrm{rev}_1\, \mathcal{N}_P^{\mu}(0) = \left [ \begin{array}{ccccc|cccc} P_k & 0 & 0 & \cdots & 0 & 0 & \cdots & 0\\
 & & & & 0 & I_n & \cdots & 0\\
&&&& \vdots &\vdots  & \ddots & \vdots \\
&&&& 0 & 0 & \cdots & I_n\\
\hline
0 & I_n & 0 & \cdots & 0 &&&\\
0 & 0 & I_n & \cdots & 0 &&&\\
\vdots & \vdots & \vdots & \ddots & \vdots & \\
0& 0 & 0 & \cdots & I_n 
\end{array}\right ].
\]
Hence, any right eigenvector $z$ of $\mathrm{rev}_1\, \mathcal{N}_P^{\mu}(\lambda)$ with eigenvalue 0 is necessarily of the form $\begin{bmatrix} x^T & 0 &  \cdots & 0\end{bmatrix}^T$ for some eigenvector $x$ of $P_k$ with eigenvalue 0.
Conclusively, the first block-entry of $z$, when seen as a block vector of length $k$, is an eigenvector of $P(\lambda)$ with eigenvalue infinity. 

Let us now prove the results for any Newton linearization $N(\lambda)$.
By Remark \ref{remark:Newton linearization form}, we have
\begin{equation}\label{eq:factored Newton lin}
N(\lambda) = \begin{bmatrix}
I_{(\mu+1)n} & A \\
0 & I_{(k-\mu-1)n} \\
\end{bmatrix}
\mathcal{N}_P^\mu(\lambda)
\begin{bmatrix}
I_{(k-\mu)n} & 0 \\
B & I_{\mu n} 
\end{bmatrix}.
\end{equation}
for some matrices $A$ and $B$.
The equivalence transformation \eqref{eq:factored Newton lin} implies that $z$ is a right eigenvector of $N(\lambda)$ with eigenvalue (finite or infinite) $\lambda_0$ if and only if $\widetilde{z} := \left[\begin{smallmatrix} I_{(k-\mu)n} & 0 \\
B & I_{\mu m} \end{smallmatrix}  \right]z$ is an eigenvector of $\mathcal{N}_P^\mu(\lambda)$ with eigenvalue (finite or infinite) $\lambda_0$.
To finish the proof, it suffices to notice that the first $k-\mu$ blocks of the eigenvectors $z$ and $\widetilde{z}$ are the same.
\end{proof}

\subsection{Recovery of minimal bases and minimal  indices from Newton linearizations}

Assume the $m\times n$ matrix polynomial $P(\lambda)=\sum_{i=0}^k P_i\,n_i(\lambda)$ is  singular.
In this section, we show  how to recover the minimal indices and minimal bases of $P(\lambda)$ from those of its Newton linearizations. 

\begin{theorem}[Recovery of minimal bases and minimal indices from Newton linearizations]
Let $P(\lambda)=\sum_{i=0}^k P_i\, n_i(\lambda)$ be an $m\times n$ singular matrix polynomial expressed in the Newton basis associated with nodes $\{x_1,\hdots,x_k\}$.
Let $0\leq \mu \leq k-1$ be an integer, and let $N(\lambda)$ be a Newton  linearization of $P(\lambda)$ as in \eqref{eq:Newton linearization}.
\begin{itemize}
\item[\rm (a1)] Suppose that $\{z_1(\lambda), z_2(\lambda), \ldots, z_p(\lambda)\}$ is a minimal basis for the right nullspace of $N(\lambda)$, with vector polynomials $z_i$ partitioned into blocks conformable with the blocks of $N(\lambda)$, and let $x_{\ell}(\lambda)$ be the $(k-\mu)$th block-entry of $z_{\ell}(\lambda)$, for $\ell =1,2, \ldots, p$.
Then, $\{x_1(\lambda), \ldots, x_p(\lambda)\}$ is a minimal basis for the right nullspace of $P(\lambda)$. 
\item[\rm (a2)] If $0 \leq \epsilon_1 \leq \cdots \leq \epsilon_p$ are the right minimal indices of $N(\lambda)$, then
\[
0 \leq \epsilon_1- k + \mu + 1 \leq \epsilon_2 -k +\mu + 1 \leq \cdots \leq \epsilon_p - k + \mu + 1
\]
are the right minimal indices of $P(\lambda)$. 
\item[\rm(b1)] Suppose that $\{\omega_1(\lambda), \ldots, \omega_q(\lambda)\}$ is a minimal basis for the left nullspace of $ N(\lambda)$, with vectors $\omega_i$ partitioned into blocks conformable with the blocks of $N(\lambda)$, and let $y_{\ell}(\lambda)$ be the $(\mu+1)$th block-entry of $\omega_{\ell}(\lambda)$, for $\ell =1,2, \ldots, q$.
 Then, $\{y_1(\lambda), \ldots, y_q(\lambda)\}$ is a minimal basis for the left nullspace of $P(\lambda)$. 
\item[\rm(b2)] If $0 \leq \mu_1 \leq \cdots \leq \mu_q$ are the left minimal indices of $N(\lambda)$, then
\[
0 \leq \mu_1 - \mu \leq \mu_2 - \mu  \leq \cdots \leq \mu_p - \mu
\]
are the left minimal indices of $P(\lambda)$. 
\end{itemize}

\end{theorem}

\begin{proof}
The proof follows closely the proof of Theorem \ref{recover-eig}, so we just sketch it. 
First, using Theorem \ref{thm:factorizations} together with the one-sided factorizations in Theorem \ref{thm:Newton factorizations} one proves the results for the Newton colleague pencil \eqref{eq:Newton colleague}.
Then, using the strict equivalence 
\[
N(\lambda) = \begin{bmatrix}
I_{(\mu+1)m} & A \\
0 & I_{(k-\mu-1)n} \\
\end{bmatrix}
\mathcal{N}_P^\mu(\lambda)
\begin{bmatrix}
I_{(k-\mu)n} & 0 \\
B & I_{\mu m} 
\end{bmatrix},
\]
that transform the Newton colleague pencil into the Newton linearization $N(\lambda)$, one proves the result for $N(\lambda)$.
\end{proof}

\section{Strong linearizations for matrix polynomials in the Lagrange basis}\label{sec:Lagrange}
Let  $\{x_1,\hdots,x_{k+1}\}$ be a set of $k+1$ nodes, and let $P(\lambda)$ be a matrix polynomial expressed in the modified Lagrange form:
\begin{equation}\label{pol-lag}
P(\lambda)=\ell(\lambda)\sum\limits_{i=1}^{k+1} P_i\,\frac{w_i}{\gamma_i(\lambda)}, \quad P_1,\hdots,P_{k+1}\in\mathbb{C}^{m\times n},
\end{equation}
where $\gamma_i(\lambda) = \lambda-x_i$, and $\ell(\lambda)$ and $w_i$ are as in \eqref{weights}.
In this section, we  preset  a family of strong linearizations of the polynomial $P(\lambda)$ that can be easily constructed from the coefficients $P_i$  and the corresponding  nodes.

Let $ 0\leq \mu \leq k-1$ be an integer.
We define the following matrix pencils
\begin{equation}\label{K1L}K_1^{L}(\lambda):=\left[\begin{array}{c c c c c }
    \gamma_{k+1}(\lambda) I_n & -\gamma_{k-1}(\lambda) I_n &&&\\
    & \gamma_{k}(\lambda) I_n & -\gamma_{k-2}(\lambda) I_n &&\\
    && \ddots & \ddots &\\
    &&& \gamma_{\mu+3}(\lambda) I_n & -\gamma_{\mu+1}(\lambda) I_n
\end{array}\right]
\end{equation}
and 
\begin{equation}\label{K2L}K_2^L(\lambda):=\left[\begin{array}{c c c c c}
    \gamma_{\mu+2}(\lambda)I_m & -\gamma_{\mu}(\lambda)I_m &&&\\
    & \gamma_{\mu+1}(\lambda)I_m & -\gamma_{\mu-1}(\lambda)I_m &&\\
    && \ddots & \ddots &\\
    &&& \gamma_{3}(\lambda)I_m & -\gamma_{1}(\lambda)I_m
\end{array}\right],
\end{equation}
where the polynomials $\gamma_j(\lambda)$ are defined in \eqref{eq:gamma}.
Notice that when $\mu=0$ (resp. $\mu=k-1$), the matrix pencil $K_2^L(\lambda)$ (resp. $K_1^L(\lambda)$) is an empty matrix.

\begin{lemma}
Let $\{x_1,\hdots,x_{k+1}\}$ be a set of nodes, and let $0\leq\mu\leq k-1$ be an integer.
The matrix pencils $K_1^L(\lambda)$ and $K_2^L(\lambda)$ given in \eqref{K1L} and \eqref{K2L} are both minimal bases. 
Moreover, the matrix polynomials 
\begin{equation}\label{D1D2L}
D_1^L(\lambda)^T=\left[\begin{array}{c}
    \dfrac{n_{\mu+1}^{k+1}(\lambda)}{\gamma_{k+1}(\lambda)\gamma_{k}(\lambda)}I_n\\[10pt] \dfrac{n_{\mu+1}^{k+1}(\lambda)}{\gamma_{k}(\lambda)\gamma_{k-1}(\lambda)}I_n\\
    \vdots\\ \dfrac{n_{\mu+1}^{k+1}(\lambda)}{\gamma_{\mu+2}(\lambda)\gamma_{\mu+1}(\lambda)} I_n
\end{array}\right]
\quad
\textrm{and}
\quad 
D_2^L(\lambda)^T=\left[\begin{array}{c}
    \dfrac{n_1^{\mu+2}(\lambda)}{\gamma_{\mu+2}(\lambda)\gamma_{\mu+1}(\lambda)}I_m\\[10pt] \dfrac{n_1^{\mu+2}(\lambda)}{\gamma_{\mu+1}(\lambda)\gamma_{\mu}(\lambda)}I_m\\ \vdots\\ \dfrac{n_1^{\mu+2}(\lambda)}{\gamma_{2}(\lambda)\gamma_{1}(\lambda)} I_m
\end{array}\right],
\end{equation}
where the polynomials $n_i^j(\lambda)$ are defined in \eqref{eq:nij}, are, respectively, dual bases of $K_1^L(\lambda)$ and $K_2^L(\lambda)$. 
\end{lemma}

\begin{proof}
It is easy to check through straightforward computations that $K_1^L(\lambda) D_1^L(\lambda)^T=0$ and $K_2^L(\lambda)D_2^L(\lambda)^T =0$. 
The minimality of the four matrix polynomials follows from the characterization of minimal bases in Theorem \ref{minimal-basis}.
\end{proof}


We now consider strong block minimal basis pencils of the form
\begin{equation}\label{eq:Lagrange pencil}
L(\lambda) = \begin{bmatrix}
M(\lambda) & K_2^L(\lambda)^T \\
K_1^L(\lambda) & 0 
\end{bmatrix}.
\end{equation}
We will refer to \eqref{eq:Lagrange pencil} as a \emph{Lagrange pencil}.
In theorem \ref{thm:Lagrange-colleague}, we show how to chose the body $M(\lambda)$  of a Lagrange pencil \eqref{eq:Lagrange pencil} so that the Lagrange pencil is a strong linearization of the matrix polynomial  \eqref{pol-lag}.

\begin{theorem}\label{thm:Lagrange-colleague}
Let  $P(\lambda)$ be an $m \times n$ matrix polynomial as in \eqref{pol-lag}. 
Let $0\leq \mu \leq k-1$ be an integer and let $M_{\mu}^L(\lambda):=$
\begin{equation*}\label{eq:Mlag}
\left[\begin{array}{c c c c}
    P_{k+1}w_{k+1}\gamma_{k}(\lambda)+P_{k}w_{k}\gamma_{k+1}(\lambda) & P_{k-1}w_{k-1}\gamma_{k}(\lambda) & \dots & P_{\mu+1}w_{\mu+1}\gamma_{\mu+2}(\lambda)\\
    &&& P_{\mu}w_{\mu}\gamma_{\mu+1}(\lambda)\\
    &&& \vdots\\
    &&& P_{2}w_{2}\gamma_{3}(\lambda)\\
    &&& P_{1}w_{1}\gamma_{2}(\lambda)
\end{array}\right],
\end{equation*}
when $0\leq \mu< k-1$; and 
\begin{equation*}\label{eq:Mlag2}
M_{\mu}^L(\lambda):=\left[ \begin{array}{c} P_{k+1} w_{k+1} \gamma_{k}(\lambda) +P_{k} w_{k} \gamma_{k+1}(\lambda) \\ P_{k-1} w_{k-1} \gamma_{k}(\lambda) \\ \vdots\\ P_2 w_2 \gamma_3(\lambda) \\ P_1 w_1 \gamma_2(\lambda) \end{array} \right],
\end{equation*}
when $\mu =k-1$. 
Then, the Lagrange pencil  
\begin{equation}\label{eq:collegue Lagrange}
L_P^\mu(\lambda) = \begin{bmatrix}
M_\mu^L(\lambda) & K_2^L(\lambda)^T \\
K_1^L(\lambda) & 0 
\end{bmatrix}.
\end{equation}
is a strong linearization of $P(\lambda)$. 
We will refer to \eqref{eq:collegue Lagrange} as the \emph{colleague Lagrange pencil} of $P(\lambda)$ associated with $\mu$. 
\end{theorem}
\begin{proof}
By direct matrix multiplication, we have $D_2^L(\lambda)M_\mu^L(\lambda)D_1^L(\lambda)^T=P(\lambda)$, where   $D_1^L(\lambda)$ and $D_2^L(\lambda)$ are the dual minimal basis of $K_1^L(\lambda)$ and $K_2^L(\lambda)$, respectively.
Thus, by Theorem \ref{thm:key1}, the colleague Lagrance pencil $L_P^\mu(\lambda)$ is a strong linearization of the matrix polynomial $P(\lambda)$.
\end{proof}

\begin{remark}
Previously to this work, and as far as we know, the only strong linearization for matrix polynomials in the Lagrange basis as in \eqref{pol-lag} of size $nk\times nk$ explicitly constructed is 
\begin{equation}\label{lin-lag}
\left[\begin{array}{c c c c c}
 -\gamma_1 P_0 & -\gamma_2 P_1 & \dots & -\gamma_{k-1} P_{k-2} & -\gamma_k P_{k-1}- \gamma_{k-1} \theta^{-1}_k P_k\\
    \hline
   -\gamma_0 I & \gamma_2 \theta_1I & & & \\
     & \ddots & \ddots & & \\
     & & -\gamma_{k-3}I & \gamma_{k-1} \theta_{k-2}I & \\
     & & & -\gamma_{k-2} I & \gamma_k \theta_{k-1}I
\end{array}\right],
\end{equation}
where $\theta_i =\frac{w_{i-1}}{w_i}$, for  $i=1,\hdots,k$, and where we omit the dependence on $\lambda$ of the $\gamma_i$ polynomials for lack of space. 
This strong linearization was introduced in \cite{Lagrange}, and it can be easily established to be strictly equivalent to the Lagrange colleague pencil \eqref{eq:collegue Lagrange} associated with $\mu=0$.
\end{remark}

By applying Theorem \ref{thm:key2} to the colleague Lagrance pencil \eqref{eq:collegue Lagrange}, we construct in Theorem \ref{thm:Lagrange linearizations} an infinite family of strong linearizations of a matrix polynomial $P(\lambda)$ expressed in the Lagrange basis.
\begin{theorem}\label{thm:Lagrange linearizations}
Let  $P(\lambda)$ be a matrix polynomial expressed in the Lagrange basis as in \eqref{pol-lag}. 
Let $0\leq \mu \leq k-1$ be an integer and let $M_{\mu}^L$ be as in Theorem \ref{thm:Lagrange-colleague}.
 Let $A$ and $B$ be two arbitrary matrices of size $(\mu+1)m\times (k-\mu-1)n$ and $\mu m \times (k-\mu)n$, respectively. Then,  the pencil 
\begin{equation}\label{eq:Lagrange linearization}
L(\lambda) := 
\left[ \begin{array}{c|c}
   M_{\mu}^L(\lambda) +AK_1^L(\lambda)+K_2^L(\lambda)^TB & K_2^L(\lambda)^T \\
    \hline
   K_1^L(\lambda)  & 0
\end{array} \right]
\end{equation}
is a strong linearization of $P(\lambda)$. We will refer to \eqref{eq:Lagrange linearization} as a \emph{Lagrange linearization} of the matrix polynomial $P(\lambda)$.
\end{theorem}
\begin{remark}\label{remark:eqLpmu}
Note that every Lagrange linearization \eqref{eq:Lagrange linearization} of a matrix polynomial $P(\lambda)$ is strictly equivalent to the colleague Lagrange pencil $L(\lambda)$ as in \eqref{eq:collegue Lagrange}, since we have
\[
L(\lambda)=
\begin{bmatrix}
I_{(\mu+1)m} & A \\
0 & I_{(k-\mu-1)n}
\end{bmatrix}
 \begin{bmatrix}
M_\mu^L(\lambda) & K_2^L(\lambda)^T \\
K_1^L(\lambda) & 0 
\end{bmatrix}
\begin{bmatrix}
I_{(k-\mu)n} & 0 \\ B & I_{\mu m}
\end{bmatrix}.
\]
\end{remark}



Next we construct a few examples of Lagrange linearizations of a matrix polynomial of grade 5. 

\begin{example}
Let $P(\lambda)$ be a matrix polynomial expressed in the Lagrange basis as in \eqref{pol-lag} of grade 5.
 Let $\mu=2$. 
 Then, the Lagrange colleague pencil of $P(\lambda)$ is given by $L_P^2(\lambda)=$
\[
 \left[ \begin{array}{ccc|cc}  P_6 w_6 \gamma_5+P_5 w_5 \gamma_6 &  P_4 w_4 \gamma_5& P_3 w_3 \gamma_4 & \gamma_4 I_m & 0 \\ 
0&0& P_2 w_2 \gamma_3 &- \gamma_2 I_m & \gamma_3 I_m \\0 &0& P_1 w_1 \gamma_2 & 0 &- \gamma_1 I_m\\
\hline
\gamma_6 I_n & -\gamma_4 I_n & 0 & 0 & 0 \\ 
0& \gamma_5 I_n & -\gamma_3 I_n & 0 & 0  
\end{array} \right ],
\]
where, for lack of space, we omit the dependence in $\lambda$ of the $\gamma_i(\lambda)$ polynomials.
By Theorem \ref{thm:Lagrange linearizations}, the following Lagrange pencils are also strong linearizations of $P(\lambda)$. 
They are obtained from the Lagrange colleague pencil by applying a finite number of elementary block-row or block-column operations, in the same spirit as in Example \ref{ex:Newton linearizations}.
 Using the notation in Theorem \ref{thm:Lagrange linearizations}, we specify the matrices $A$ and $B$ used to obtain the body of each particular linearization.

The following linearization has been obtained from $L_P^2(\lambda)$ by adding to the first block-row the fifth block-row multiplied by $-P_4w_4$: 
 \[
L_1(\lambda)= \left[ \begin{array}{ccc|cc}  P_6 w_6 \gamma_5+P_5 w_5 \gamma_6 & 0 &P_4w_4\gamma_3+ P_3 w_3 \gamma_4 & \gamma_4 I_m & 0 \\ 
0 & 0 & P_2 w_2 \gamma_2 &- \gamma_2 I_m & \gamma_3 I_m \\0 & 0 & P_1 w_1 \gamma_2 & 0 &- \gamma_1 I_m\\
\hline
\gamma_6I_n & -\gamma_4 I_n & 0 & 0 & 0 \\ 
0 & \gamma_5 I_n & -\gamma_3 I_n & 0 & 0  
\end{array} \right ],  
\]
In this case, we have
$A = \left [\begin{smallmatrix} 0 & -P_4 w_4 \\ 0 &0 \\ 0 & 0  \end{smallmatrix}  \right]$ and $B = 0$.

The following linearization has been obtained from $L_1(\lambda)$ by adding to the first block-row the fourth block-row multiplied by $- P_5w_5 $: 
\[
L_2(\lambda)= \left[ \begin{array}{ccc|cc}  P_6 w_6 \gamma_5&  P_5 w_5 \gamma_4 & P_4 w_4 \gamma_3+ P_3 w_3 \gamma_4 & \gamma_4 I_m & 0 \\ 
0 & 0 & P_2 w_2 \gamma_3 &- \gamma_2 I_m & \gamma_3 I_m \\0 &0& P_1 w_1 \gamma_2 & 0 &- \gamma_1 I_m\\
\hline
\gamma_6 I_n & -\gamma_4 I_n & 0 & 0 & 0 \\ 
0 & \gamma_5 I_n & -\gamma_3 I_n & 0 & 0 
\end{array} \right ].
\]
In this case, we have
$A = \left [\begin{smallmatrix} -P_5w_5 & -P_4 w_4 \\ 0 &0 \\ 0 & 0  \end{smallmatrix}  \right]$ and $B = 0$.

The following linearization has been obtained from $L_2(\lambda)$ by adding to the second block-column the fourth block-column multiplied by $- P_5 w_5$:
\[
L_3(\lambda)= \left[ \begin{array}{ccc|cc}  P_6 w_6 \gamma_5  & 0 &P_4 w_4 \gamma_3+ P_3 w_3 \gamma_4 & \gamma_4 I_m & 0 \\ 
0 & P_5 w_5 \gamma_2 & P_2 w_2 \gamma_3 &- \gamma_2 I_m & \gamma_3 I_m \\0 &0& P_1 w_1 \gamma_2 &0 &- \gamma_1 I_m\\
\hline
\gamma_6 I_n & -\gamma_4 I_n & 0 & 0 & 0 \\ 
0 & \gamma_5 I_n & -\gamma_3 I_n & 0 & 0  
\end{array} \right ].
\]
In this case, we have
$A = \left [\begin{smallmatrix} -P_5w_5 & -P_4 w_4 \\ 0 &0 \\ 0 & 0  \end{smallmatrix}  \right]$ and $B = \left [\begin{smallmatrix} 0 & -P_5 w_5  & 0\\ 0 &0 & 0 \end{smallmatrix}  \right]$.
\end{example}

Our next goal is to obtain recovery rules for eigenvectors, and minimal bases and minimal indices of a matrix polynomial $P(\lambda)$ from those of its Lagrange linearizations.
We will need the following notation.

Associated with the matrix polynomial $P(\lambda)$ in \eqref{pol-lag}, we define the matrix polynomials
\[
T_{j}^P(\lambda) :=  \ell(\lambda)  \sum_{i=1}^j P_i \frac{w_i}{\gamma_i(\lambda)} \quad \mbox{and} \quad S_j^P(\lambda) := \ell(\lambda) \sum_{i=j}^{k+1} P_i \frac{w_i}{\gamma_i(\lambda)} \quad (j=1,\hdots,k+1),
\]
where, we recall, $\ell(\lambda)=n_1^{k+1}(\lambda)=\prod_{i=1}^{k+1}(\lambda-x_i)$. 
Observe that $T_{k+1}^P(\lambda)=S_1^P(\lambda)= P(\lambda)$.
Moreover, we have 
\[
S_{j+1}^P(\lambda)+T_{j}^P(\lambda)=P(\lambda),\quad  \textrm{for $j=1, 2,\ldots, k$}.
\]
 Let $0\leq \mu \leq k-1$ be an integer and let  $a_{\mu+1}, \ldots , a_2,  a_{1}$ be the coordinates of the (scalar) polynomial $p(x)=1$ ``in the basis $D_2^L(\lambda)$'', that is,
\begin{equation}\label{mu-2coordinates}
a_{\mu+1} \frac{n_1^{\mu+2}(\lambda)}{\gamma_{\mu+2}(\lambda) \gamma_{\mu+1}(\lambda)} + a_{\mu} \frac{n_1^{\mu+2}(\lambda)}{\gamma_{\mu+1}(\lambda)\gamma_{\mu}(\lambda)} + \cdots + a_2 \frac{n_1^{\mu+2}(\lambda)}{\gamma_{3}(\lambda) \gamma_{2}(\lambda)} + a_1 \frac{n_1^{\mu+2}(\lambda)}{\gamma_{2}(\lambda) \gamma_{1}(\lambda)} =1.
\end{equation}
We call $[a_{\mu+1}, a_{\mu}, \ldots, a_2, a_1]$ the \emph{$\mu$-2-coordinates of $1$}.  
We notice that, by evaluating the expression \eqref{mu-2coordinates} at the nodes $x_1$ and $x_{\mu+2}$, respectively, we get the values of $a_1$ and $a_{\mu+2}$, namely,
\[
a_1 = \frac{1}{ \prod_{i=3}^{\mu+2} (x_1-x_i)} \quad \mbox{and} \quad a_{\mu+1}= \frac{1}{ \prod_{i=1}^{\mu}(x_{\mu+2}-x_i)}.
\]
The rest of the coordinates can be obtained from the  recurrence relation
\[
1= a_{i}\left. \frac{n_1^{\mu+2}(\lambda)}{\gamma_{i+1}(\lambda)\gamma_{i}(\lambda)}\right|_{\lambda=x_i} + a_{i-1} \left.\frac{n_1^{\mu+2}(\lambda)}{\gamma_{i}(\lambda)\gamma_{i-i}(\lambda)}\right|_{\lambda=x_i},
\]
which is the result of evaluating \eqref{mu-2coordinates} at the node $x_i$ ($i=2,\hdots,\mu+1$).

Similarly, let $b_{\mu+1}, b_{\mu+2}, \ldots, b_{k}$ be the coordinates of the polynomial $p(x)=1$ ``in the basis $D_1^L(\lambda)$'', that is,
\begin{equation}\label{mu-1coordinates}
b_{k} \frac{n_{\mu+1}^{k+1}(\lambda)}{\gamma_{k+1}(\lambda)\gamma_{k}(\lambda)} + \ldots + b_{\mu+2} \frac{n_{\mu+1}^{k+1}(\lambda)}{\gamma_{\mu+3}(\lambda)\gamma_{\mu+2}(\lambda)} + b_{\mu+1} \frac{n_{\mu+1}^{k+1}(\lambda)}{\gamma_{\mu+2}(\lambda)\gamma_{\mu+1}(\lambda)} =1
\end{equation}
We call $[b_{k}, \ldots, b_{\mu+2}, b_{\mu+1}]$ the $\mu$-1-coordinates of 1. The numbers $b_i$ can be obtained using the same approach used to compute the $\mu$-2-coordinates of 1. 

Finally, we  denote 
\begin{align*}
\mathcal{P}_{j}^{\mu}(\lambda)&:=
- \sum_{i=1}^{j} a_i \frac{\gamma_{j+1}(\lambda)}{\gamma_i(\lambda)\gamma_{i+1}(\lambda)} S_{j+1}^P(\lambda)+
\sum_{i={j+1}}^{\mu+1} a_i \frac{\gamma_{j+1}(\lambda)}{\gamma_i(\lambda)\gamma_{i+1}(\lambda)} T_{j}^P(\lambda), 
\end{align*}
for $j=1, 2, \ldots, \mu$, and
\begin{align*}
\mathcal{Q}_j^{\mu}(\lambda)&:=
- \sum_{i=\mu+1}^{j} b_i \frac{\gamma_{j+1}(\lambda) }{\gamma_{i}(\lambda) \gamma_{i+1}(\lambda)}  S_{j+1}^P(\lambda) + \sum_{i=j+1}^{k} b_i \frac{\gamma_{j+1}(\lambda)}{\gamma_{i}(\lambda)\gamma_{i+1}(\lambda)} T_{j}^P(\lambda),
\end{align*}
for $j=\mu+1, \ldots, k-1$.
We observe that both $\mathcal{P}_j^\mu(\lambda)$ ($j=1,\hdots,\mu$) and $\mathcal{Q}_j^\mu(\lambda)$ ($j=\mu+1,\hdots,k-1$) are matrix polynomials.

Theorem \ref{thm:one-sided-lag} gives right- and left-sided factorizations of the Lagrange colleague pencil \eqref{eq:collegue Lagrange}.
\begin{theorem}\label{thm:one-sided-lag}
Let $P(\lambda) $ be a matrix polynomial of degree $k$ as in \eqref{pol-lag}, let $0 \leq \mu \leq k-1$ be an integer, let $L_P^\mu(\lambda)$ be the Lagrange colleague pencil in \eqref{eq:collegue Lagrange}, and let $D_1^L(\lambda)$ and $D_2^L(\lambda)$ be the minimal bases in \eqref{D1D2L}.

 For  $0< \mu \leq  k-1$,  let
\[
H_L^{\mu}(\lambda)^{\mathcal{B}}:=
\begin{bmatrix}
D_1^L(\lambda) & \mathcal{P}_{\mu}^{\mu}(\lambda)& \mathcal{P}_{\mu-1}^{\mu}(\lambda) & \cdots &  \mathcal{P}_1^{\mu}(\lambda)
\end{bmatrix}
\]
and for $\mu=0$, let $H_L^{\mu}(\lambda)^{\mathcal{B}}:= D_1^L(\lambda)$.
 
For $0\leq \mu < k-1$, let 
\[
G_L^{\mu}(\lambda)^{\mathcal{B}}:=  
\begin{bmatrix}
D_2^L(\lambda) & \mathcal{Q}^{\mu}_{k-1}(\lambda) & \cdots & Q_{\mu+1}^{\mu}(\lambda)
 \end{bmatrix},
\]
and for $\mu=k-1$, let
$G_L^{\mu}(\lambda)^{\mathcal{B}}:= D_2^L(\lambda)$.
Then, the following right- and left-sided factorizations hold
$$L_P^{\mu}(\lambda) H_L^{\mu}(\lambda) =\left( \sum_{i=1}^{\mu+1}  a_{\mu+2-i} e_{i}\right) \otimes P(\lambda) \quad \textrm{and} \quad  G_L^{\mu}(\lambda)^{\mathcal{B}} L_P^{\mu}(\lambda)= \left (\sum_{i=\mu+1}^{k} b_i e_i^T\right) \otimes P(\lambda),$$
where $e_i$ denotes the $i$th column of the $k\times k$ identity matrix, and where $\begin{bmatrix} a_{\mu+1} &  \cdots & a_1 & a_1\end{bmatrix}$ and $\begin{bmatrix} b_k &  \cdots & b_{\mu+2} & b_{\mu+1} \end{bmatrix}$ are, respectively, the $\mu$-2-coordinates and $\mu$-1-coordinates of 1. 
\end{theorem}
\begin{proof}
We prove the right-sided factorization.
The left-sided factorization can be proven
similarly.

By the duality of the minimal bases $K_1^L(\lambda)$ and $D_1^L(\lambda)$, it is clear that the $i$th block entry, with $i\in\{\mu+1,\mu+2,\hdots,k\}$, of $L_P^\mu(\lambda)H_L^\mu(\lambda)$ is zero.  

Let $i\in\{1,2,\hdots,\mu+1\}$.
We need to compute the product of the $i$th block row of $H_L^\mu(\lambda)$ and $L^\mu_P(\lambda)$.
To do this, we have to distinguish three cases:

\medskip

\noindent Case I: Let $i=1$.
By direct matrix multiplication, the product of the first block row of $H_L^\mu(\lambda)$ and $L^\mu_P(\lambda)$ is given by
\begin{align*}
& n_{\mu+1}^k(\lambda)
\sum_{i=\mu+1}^{k+1}\frac{P_iw_i}{\gamma_i(\lambda)}+\gamma_{\mu+2}(\lambda)\mathcal{P}_\mu^\mu(\lambda) = 
\frac{S_{\mu+1}^P(\lambda)}{n_1^\mu(\lambda)}+\gamma_{\mu+2}(\lambda)\mathcal{P}_\mu^\mu(\lambda) = \\ & \hspace{3cm}
\frac{S_{\mu+1}^P(\lambda)}{n_1^\mu(\lambda)}-\sum_{i=1}^\mu \frac{a_i\gamma_{\mu+1}(\lambda)\gamma_{\mu+2}(\lambda)}{\gamma_i(\lambda)\gamma_{i+1}(\lambda)}S_{\mu+1}^P(\lambda)+a_{\mu+1}T_\mu^P(\lambda) = \\ &\hspace{3cm}
\frac{S_{\mu+1}^P(\lambda)}{n_1^\mu(\lambda)}-\sum_{i=1}^\mu \frac{a_in_1^{\mu+2}(\lambda)}{\gamma_i(\lambda)\gamma_{i+1}(\lambda)}\frac{S_{\mu+1}^P(\lambda)}{n_1^\mu(\lambda)}+a_{\mu+1}T_\mu^P(\lambda)=\\&\hspace{3cm}
\frac{S_{\mu+1}^P(\lambda)}{n_1^\mu(\lambda)}-\left( 1 - \frac{a_{\mu+1}}{\gamma_{\mu+1}(\lambda)\gamma_{\mu+2}(\lambda)}n_1^{\mu+2}(\lambda)\right)\frac{S_{\mu+1}^P(\lambda)}{n_1^\mu(\lambda)}+a_{\mu+1}T_\mu^P(\lambda) = \\ &\hspace{3cm}
a_{\mu+1}\left(S_{\mu+1}^P(\lambda)+T_\mu^P(\lambda)\right) = a_{\mu+1}P(\lambda),
\end{align*}
which is the desired result.

\medskip

\noindent Case II: Let $i\in\{2,3,\hdots,\mu\}$, and let $r=\mu+2-i$.
 The product of the $i$th block row of $H_L^\mu(\lambda)$ and $L^\mu_P(\lambda)$ is given by
 \begin{align}\label{eq:proof1}
 \begin{split}
 & P_rw_r\gamma_{r+1}(\lambda)\frac{n_{\mu+1}^{k+1}(\lambda)}{\gamma_{\mu+1}(\lambda)\gamma_{\mu+2}(\lambda)} - \gamma_r(\lambda)\mathcal{P}_r^\mu(\lambda)+\gamma_{r+1}(\lambda)\mathcal{P}_{r-1}^\mu(\lambda) = \\
 & P_rw_r\gamma_{r+1}n^{k+1}_{\mu+3}(\lambda) - 
 \\ &\hspace{1cm}
 \gamma_r(\lambda)\left(-\sum_{i=1}^ra_i\frac{\gamma_{r+1}(\lambda)}{\gamma_i(\lambda)\gamma_{i+1}(\lambda)}S_{r+1}^P(\lambda) + \sum_{i=r+1}^{\mu+1}a_i \frac{\gamma_{r+1}(\lambda)}{\gamma_i(\lambda)\gamma_{i+1}(\lambda)}T_r^P(\lambda) \right) + 
  \\ &\hspace{2cm}
   \gamma_{r+1}(\lambda)\left(-\sum_{i=1}^{r-1}a_i\frac{\gamma_{r}(\lambda)}{\gamma_i(\lambda)\gamma_{i+1}(\lambda)}S_{r}^P(\lambda) + \sum_{i=r}^{\mu+1}a_i \frac{\gamma_{r}(\lambda)}{\gamma_i(\lambda)\gamma_{i+1}(\lambda)}T_{r-1}^P(\lambda) \right).
\end{split}
 \end{align}
 Taking into account that $S_r^P(\lambda)=S_{r+1}^P(\lambda)+n_1^{k+1}(\lambda)P_rw_r/\gamma_r(\lambda)$, we get
 \begin{align}\label{eq:proof2}
 \begin{split}
 &\sum_{i=1}^r a_i \frac{\gamma_r(\lambda)\gamma_{r+1}(\lambda)}{\gamma_i(\lambda)\gamma_{i+1}(\lambda)}S_{r+1}^P(\lambda)-
 \sum_{i=1}^{r-1}a_i\frac{\gamma_{r}(\lambda)\gamma_{r+1}(\lambda)}{\gamma_i(\lambda)\gamma_{i+1}(\lambda)}S_{r}^P(\lambda) = \\
 &\sum_{i=1}^r a_i \frac{\gamma_r(\lambda)\gamma_{r+1}(\lambda)}{\gamma_i(\lambda)\gamma_{i+1}(\lambda)}S_{r+1}^P(\lambda)-
 \sum_{i=1}^{r-1}a_i\frac{\gamma_{r}(\lambda)\gamma_{r+1}(\lambda)}{\gamma_i(\lambda)\gamma_{i+1}(\lambda)}\left(S_{r+1}^P(\lambda)+n_1^{k+1}(\lambda)P_r\frac{w_r}{\gamma_r(\lambda)} \right) = \\
 &a_rS_{r+1}^P(\lambda)-\sum_{i=1}^{r-1}a_iP_rw_rn_1^{k+1}(\lambda)\frac{\gamma_{r+1}(\lambda)}{\gamma_i(\lambda)\gamma_{i+1}(\lambda)}.
 \end{split}
 \end{align}
Taking into account that $T_r^P(\lambda)=T_{r-1}^P(\lambda)+n_1^{k+1}(\lambda)P_rw_r/\gamma_r(\lambda)$, we obtain
\begin{align}\label{eq:proof3}
\begin{split}
&-\sum_{i=r+1}^{\mu+1}a_i\frac{\gamma_r(\lambda)\gamma_{r+1}(\lambda)}{\gamma_i(\lambda)\gamma_{i+1}(\lambda)}T_r^P(\lambda) + \sum_{i=r}^{\mu+1}a_i\frac{\gamma_r(\lambda)\gamma_{r+1}(\lambda)}{\gamma_i(\lambda)\gamma_{i+1}(\lambda)}T_{r-1}^P(\lambda) = \\
&-\sum_{i=r+1}^{\mu+1}a_i\frac{\gamma_r(\lambda)\gamma_{r+1}(\lambda)}{\gamma_i(\lambda)\gamma_{i+1}(\lambda)}\left( T_{r-1}^P(\lambda)+n_1^{k+1}(\lambda)P_r\frac{w_r}{\gamma_r(\lambda)}\right) + \sum_{i=r}^{\mu+1}a_i\frac{\gamma_r(\lambda)\gamma_{r+1}(\lambda)}{\gamma_i(\lambda)\gamma_{i+1}(\lambda)}T_{r-1}^P(\lambda) =\\
&a_rT_{r-1}^P(\lambda)-\sum_{i=r+1}^{\mu+1}a_iP_rw_rn_1^{k+1}(\lambda)\frac{\gamma_{r+1}(\lambda)}{\gamma_i(\lambda)\gamma_{i+1}(\lambda)}.
\end{split}
\end{align}
Substituting \eqref{eq:proof2} and \eqref{eq:proof3} into \eqref{eq:proof3} yields
\begin{align*}
 & P_rw_r\gamma_{r+1}(\lambda)\frac{n_{\mu+1}^{k+1}(\lambda)}{\gamma_{\mu+1}(\lambda)\gamma_{\mu+2}(\lambda)} - \gamma_r(\lambda)\mathcal{P}_r^\mu(\lambda)+\gamma_{r+1}(\lambda)\mathcal{P}_{r-1}^\mu(\lambda) = \\
 & P_rw_r\gamma_{r+1}(\lambda)n_{\mu+3}^{k+1}(\lambda) + a_rS_{r+1}^P(\lambda)+a_rT_{r-1}^P(\lambda) + \\
 &\hspace{3.1cm} \left(\sum_{i=1}^{r-1}\frac{a_in_1^{\mu+2}(\lambda)}{\gamma_i(\lambda)\gamma_{i+1}(\lambda)}+\sum_{i=r+1}^{\mu+1}\frac{a_in_1^{\mu+2}(\lambda)}{\gamma_i(\lambda)\gamma_{i+1}(\lambda)} \right)n_{\mu+3}^{k+1}(\lambda)\gamma_{r+1}(\lambda)P_rw_r = \\
& P_rw_r\gamma_{r+1}(\lambda)n_{\mu+3}^{k+1}(\lambda)+a_rS_{r+1}^P(\lambda)+a_rT_{r-1}^P(\lambda)+\\
&\hspace{5cm}\left(1-\frac{a_rn_1^{\mu+2}(\lambda)}{\gamma_r(\lambda)\gamma_{r+1}(\lambda)}\right)n_{\mu+3}^{k+1}(\lambda)\gamma_{r+1}(\lambda)P_rw_r = \\
&a_rS_{r+1}^P(\lambda)+a_rT_{r-1}^P(\lambda)+a_r n_{1}^{k+1}(\lambda)P_r\frac{w_r}{\gamma_r(\lambda)} = a_r\left(S_{r+1}^P(\lambda)+T_r^P(\lambda)\right) = a_r P(\lambda),
\end{align*}
as we wanted to show.
\medskip

\noindent Case III: Let $i=\mu+1$.
The product of the $(\mu+1)$th block row of $H_L^\mu(\lambda)$ and $L^\mu_P(\lambda)$ is given by
\begin{align*}
&P_1w_1\gamma_2(\lambda)\frac{n_{\mu+1}^{k+1}(\lambda)}{\gamma_{\mu+2}(\lambda)\gamma_{\mu+1}(\lambda)}-\gamma_1(\lambda)\mathcal{P}_1^\mu(\lambda)  =
\frac{T_1^P(\lambda)}{n_3^{\mu+2}(\lambda)} -\gamma_1(\lambda)\mathcal{P}_1^\mu(\lambda) = \\ & \hspace{5cm}
\frac{T_1^P(\lambda)}{n_3^{\mu+2}(\lambda)}+a_1S_2^P(\lambda)-\sum_{i=2}^{\mu+1}\frac{a_i\gamma_1(\lambda)\gamma_2(\lambda)}{\gamma_i(\lambda)\gamma_{i+1}(\lambda)} T_1^P(\lambda) = \\ & \hspace{5cm}
\frac{T_1^P(\lambda)}{n_3^{\mu+2}(\lambda)}+a_1S_2^P(\lambda)-\sum_{i=2}^{\mu+1}\frac{a_in_1^{\mu+2}(\lambda)}{\gamma_i(\lambda)\gamma_{i+1}(\lambda)} \frac{T_1^P(\lambda)}{n_3^{\mu+2}(\lambda)} = \\ &\hspace{5cm}
\frac{T_1^P(\lambda)}{n_3^{\mu+2}(\lambda)}+a_1S_2^P(\lambda)-\left(1-\frac{a_1n_1^{\mu+2}(\lambda)}{\gamma_1(\lambda)\gamma_2(\lambda)}\right)\frac{T_1^P(\lambda)}{n_3^{\mu+2}(\lambda)} = \\ &\hspace{5cm}
a_1\left(S_2^P(\lambda)+T_1^P(\lambda)\right) = a_1 P(\lambda),
\end{align*}
as we wanted to prove.
\end{proof}
\subsection{Recovery of eigenvectors from Lagrange linearizations}
Assume the matrix polynomial \eqref{eq:matrix poly in Lagrange basis} is regular.
In this section, we provide recovery formulas for the (left and right) eigenvectors of $P(\lambda)$ from those of its Lagrange linearizations.

Theorem \ref{thm:eigenvectors Lagrange linearizations} provides explicit formulas for the eigenvectors of the Lagrange colleague pencil.
\begin{theorem}\label{thm:eigenvectors Lagrange linearizations}
Let $P(\lambda)$ be a regular matrix polynomial expressed in the modified Lagrange basis associated with nodes $\{x_1,\hdots,x_{k+1}\}$.
Let $\lambda_0$ be a finite eigenvalue of $P(\lambda)$.
Let $0\leq\mu\leq k-1$ be an integer, and let $L_P^\mu(\lambda)$ be the Lagrange colleague pencil in \eqref{eq:collegue Lagrange}.
Then, $z$ (resp. $w$) is a right (resp. left) eigenvector of  $L_P^\mu(\lambda)$ associated with $\lambda_0$ if and only if $z=H_L^\mu(\lambda_0)x$ (resp. $G_L^\mu(\lambda_0)^Ty)$, where $x$ (resp. $y)$ is a right (resp. left) eigenvector of $P(\lambda)$ associated with $\lambda_0$.  
\end{theorem}
\begin{proof}
The eigenvector formulas follows from Theorems \ref{thm:factorizations} and \ref{thm:one-sided-lag}.
\end{proof}

Theorem \ref{thm:recover-eig-Lagrange} provides recovery formulas of eigenvectors (associated with finite and infinite eigenvalues) of the matrix polynomial $P(\lambda)$ from those of its Lagrange linearizations. 
We note that, in this theorem, we only consider finite eigenvalues $\lambda$ that are not an interpolation node, which is the most likely case in applications, since when $\lambda$ is a node,  many sub-cases need to be considered and make the theorem difficult to read. In any case, in Remark \ref{rem:subcases}, all those sub-cases are presented for completion.

\begin{theorem}[Recovery of eigenvectors from Lagrange linearizations]\label{thm:recover-eig-Lagrange}
Let $P(\lambda)$ be an $n \times n$ regular matrix polynomial expressed in the modified Lagrange basis as in \eqref{pol-lag},  and let $\lambda_0$ be an eigenvalue (finite or infinite) of $P(\lambda)$.  
 Let $L(\lambda)$ be a Lagrange  linearization of $P(\lambda)$ as in \eqref{eq:Lagrange linearization}.  Let $z$ and $\omega$ be, respectively, a right and a left eigenvector of $L(\lambda)$ associated with $\lambda_0$. 
\begin{enumerate}
\item Assume $\lambda_0$ is finite and $\lambda_0 \notin \{x_{1}, x_1, \ldots, x_{k+1}\}$. 
Then,
\begin{itemize}
\item the block-entries $z(1), z(2), \ldots, z(k-\mu)$ are right eigenvectors of $P(\lambda)$ associated with $\lambda_0$, and
\item  the block-entries  $\omega(1), \omega(2), \ldots, \omega(\mu+1)$  are left eigenvectors of $P(\lambda)$ associated with $\lambda_0$. 
\end{itemize}
\item Assume $\lambda_0$ is infinite.
Then,
\begin{itemize}
\item the block entries $z(1), z(2), \ldots, z(k-\mu)$  are right eigenvectors of $P(\lambda)$ associated with $\lambda_0$, and
\item the block-entries $\omega(1), \omega(2), \ldots,  \omega(\mu+1)$  are left eigenvectors of $P(\lambda)$ associated with $\lambda_0$.
 \end{itemize} 
\end{enumerate}
\end{theorem}
\begin{proof}
We prove the result for the right eigenvectors. The proof is similar for the left eigenvectors. 

We show first that the theorem holds for the Lagrange colleague pencil $L_P^\mu(\lambda)$.

Case I: Assume that $\lambda_0$ is a finite eigenvalue such that $\lambda_0 \notin \{x_{1}, x_1, \ldots, x_{k+1}\}$, and let $z$ be a right  eigenvector of the Lagrange colleague pencil $L_P^\mu(\lambda)$ associated with $\lambda_0$. 
By Theorem \ref{thm:eigenvectors Lagrange linearizations}, we have $z=H_L^\mu(\lambda_0)x$, for some right eigenvector $x$ of $P(\lambda)$ with eigenvalue $\lambda_0$.
Then, it is clear that the top $k-\mu$ block entries of $z$ are all nonzero multiples of the eigenvector $x$. 

Case II: Assume that $\lambda_0$ is an infinite eigenvalue of $P(\lambda)$.
 This means that zero is an eigenvalue of $\textrm{rev}_k P(\lambda)$ and $\textrm{rev}_1 L_P^{\mu}(\lambda)$.
By Lemma \ref{revk}, we have 
\[
\textrm{rev}_k P(\lambda)=\sum_{i=1}^{k+1} P_i \,\textrm{rev}_k \ell_i(\lambda)= \sum_{i=1}^{k+1} P_i \, \widetilde{\ell}_i(\lambda),
\]
where $\widetilde{\ell}_i(\lambda)= w_i \prod_{j=1,\,j\neq i}^k (1 - x_j \lambda)$.  Thus, $\textrm{rev}_k P(0) =\sum_{i=1}^{k+1}w_i P_i.$ 
Moreover, we also have $\textrm{rev}_1 L_P^{\mu} (0)=$
{\small $$\left[ \begin{array}{ccccc|cccc} 
P_{k+1} w_{k+1} + P_{k} w_{k} & P_{k-1} w_{k-1}& P_{k-2} w_{k-3} & \cdots & P_{\mu+1} w_{\mu+1} & I_n & 0 & \cdots &  0 \\
& & & &  P_{\mu} w_{\mu} & -I_n  & I_n & \ddots & \vdots  \\
& &&  & \vdots & 0 & \ddots & \ddots & 0 \\
& &&  & P_2 w_2 & \vdots & \ddots  & -I_n & I_n\\
&&& & P_1 w_1 &0  & \cdots & 0 & - I_n \\
\hline
I_n & -I_n & 0 &  \cdots & 0 & & & &\\
0 & I_n & - I_n & \ddots &\vdots &&&&\\
\vdots & \ddots & \ddots  & \ddots & 0 &&&&\\
0& \cdots & 0 & I_n & -I_n & 
\end{array} \right ].$$}%
From the structure of the matrix $\textrm{rev}_1 L_P^{\mu} (0)$, it follows that any right eigenvector of $\textrm{rev}_1 L_P^{\mu}(\lambda)$ with eigenvalue zero must be of the form
\[
z = 
\begin{bmatrix} x & \cdots & x & -\sum_{i=\mu+1}^{k+1} P_i w_ix & -\sum_{i=\mu}^{k+1} P_i w_ix & \cdots &  -\sum_{i=2}^{k+1} P_i w_ix 
\end{bmatrix}^\mathcal{B},
\]
for some eigenvector $x$ of $\mathrm{rev}_kP(\lambda)$ with eigenvalue zero.
Hence, we can recover $x$ from any of the top $k-\mu$ block-entries of $z$.

The results for the Lagrange linearization $L(\lambda)$ in \eqref{eq:Lagrange linearization} follows from the results for the Lagrange colleague pencil $L_P^\mu(\lambda)$ and Remark \eqref{remark:eqLpmu}.
\end{proof}
\begin{remark}\label{rem:subcases}
In the unlikely case that $\lambda_0\in\{x_1,x_2,\hdots,x_{k+1}\}$, right and left eigenvectors of $P(\lambda)$ can still be recovered from the eigenvectors of a Lagrange linearization.
With the notation used in Theorem \ref{thm:recover-eig-Lagrange}, we have:
\begin{itemize}
\item If $\lambda_0=x_1$ (resp. $\lambda_0=x_{\mu+2}$), then $z(1),\ldots, z(k- \mu)$ (resp.  $z(k-\mu-1)$ and $z(k-\mu)$) are right eigenvectors of $P(\lambda)$ associated with $\lambda_0$, and $\omega(\mu+1)$ (resp. $\omega(1)$) is a left eigenvector of $P(\lambda)$ associated with $\lambda_0$.
\item $\lambda_0=x_j\in \{x_2, \ldots, x_{\mu}\}$, then $z(1), \ldots, z(k-\mu)$ are right eigenvectors of $P(\lambda)$ associated with $\lambda_0$, and $\omega(\mu-j+2)$ and $\omega(\mu-j+3)$ are left eigenvectors of $P(\lambda)$ associated with $\lambda_0$. 
\item If $\lambda_0=x_{\mu+1}$ (resp. $\lambda_0=x_{k+1}$), then $z(k-\mu)$ (resp. $z(1)$) is a right eigenvector of $P(\lambda)$ associated with $\lambda_0$, and $\omega(1)$ and $w(2)$ (resp. $w(1), \ldots w(\mu+1)$) are left eigenvectors of $P(\lambda)$ associated with $\lambda_0$. 
\item If $\lambda_0 = x_j \in \{x_{\mu+3}, \ldots, x_k\}$, then $z(k-j+1)$ and $z(k-j+2)$ are right eigenvectors of $P(\lambda)$ associated with $\lambda_0$, and $\omega(1),\ldots, \omega(\mu+1)$ are left eigenvectors of $P(\lambda)$ associated with $\lambda_0$.
\end{itemize}
\end{remark}

\subsection{Recovery of minimal bases and minimal indices from Lagrange linearizations}
Assume the matrix polynomial $P(\lambda)$ in \eqref{pol-lag} is singular.
In this section, we show how to recover the minimal indices and minimal bases of $P(\lambda)$ from those of its Lagrange linearizations.
\begin{theorem}[Recovery of minimal bases and minimal indices from Lagrange linearizations]\label{thm:recovery-minimal}
Let $P(\lambda)$ be a singular matrix polynomial expressed in the modified Lagrange basis as in \eqref{pol-lag}.
Let $0\leq\mu\leq k-1$ be an integer, and let $L(\lambda)$ be a Lagrange linearization of $P(\lambda)$ as in \eqref{eq:Lagrange linearization}.
Let $a_{\mu+1},\hdots,a_1$ and $b_k,\hdots,b_{\mu+1}$ be, respectively, the $\mu$-1- and $\mu$-2-coordinates of 1.
\begin{itemize}
\item[\rm(a1)] Suppose that $\{z_1(\lambda),z_2(\lambda),\hdots,z_p(\lambda)\}$ is a minimal basis for the right nullspace of $L(\lambda)$, with vector polynomials $z_i$ partitioned into blocks conformable with the blocks of $L(\lambda)$.
Let 
\[
x_i(\lambda) = \begin{bmatrix}
b_k I_n & \cdots & b_{\mu+1}I_n & 0 & \cdots & 0
\end{bmatrix}z_i(\lambda) \quad (i=1,\hdots,p).
\]
Then, $\{x_1(\lambda),x_2(\lambda),\hdots,x_p(\lambda)\}$ is a minimal basis for the right nullspace of $P(\lambda)$.
\item[\rm(a2)] If $0\leq\epsilon_1\leq\cdots\leq \epsilon_p$ are the right minimal indices of  $L(\lambda)$, then
\[
0\leq\epsilon_1-k+\mu+1 \leq \cdots \leq \epsilon_p-k+\mu+1
\]
are the right minimal indices of $P(\lambda)$.
\item[\rm(b1)] Suppose that $\{w_1(\lambda),w_2(\lambda),\hdots,w_q(\lambda)\}$ is a minimal basis for the left nullspace of $L(\lambda)$, with vector polynomials $w_i$ partitioned into blocks conformable with the blocks of $L(\lambda)$.
Let
\[
y_i(\lambda) = 
\begin{bmatrix}
 a_{\mu+1} I_m & \cdots & a_1 I_m& 0 & \cdots & 0
\end{bmatrix}w_i(\lambda) \quad (i=1,\hdots,q).
\]
Then $\{y_1(\lambda),y_2(\lambda),\hdots,y_q(\lambda)\}$ is a minimal basis for the left nullspace of $P(\lambda)$.
\item[\rm(b2)] If $0\leq\mu_1\leq\cdots\leq\mu_q$ are the left minimal indices of $L(\lambda)$, then
\[
0\leq\mu_1-\mu\leq\cdots\leq\mu_q-\mu
\] 
are the left minimal indices of $P(\lambda)$.
\end{itemize}
\end{theorem}
\begin{proof}
We prove the result for the right minimal indices and bases.
The results for the left minimal indices and bases can be proven similarly.

Let $B(\lambda)$ be a matrix whose columns form a basis for the right nullspace of $P(\lambda)$.
From Theorems \ref{thm:factorizations} and \ref{thm:one-sided-lag}, we have that the columns of $H_L^\mu(\lambda)B(\lambda)$ form a basis for the right nullspace of the Lagrange colleague pencil $L_P^\mu(\lambda)$ in \eqref{eq:collegue Lagrange}.
From the definition of the $\mu$-1-coordinates of 1, we have
\[
\begin{bmatrix}
b_k I_n & \cdots & b_{\mu+1}I_n & 0 & \cdots & 0
\end{bmatrix}
 H_L^\mu(\lambda)B(\lambda)=B(\lambda).
\]
Hence, part (a1) holds for the Lagrange colleague pencil.
Part (a2) follows also from Theorems \ref{thm:factorizations} and \ref{thm:one-sided-lag}, together with the fact $\deg(D_1^L(\lambda))=k-\mu-1$, in the case that $L(\lambda)$ is the Lagrange colleague pencil.
When $L(\lambda)$ is a Lagrange linearization other than the Lagrange colleague pencil, parts (a1) and (a2) follow from Remark  \eqref{remark:eqLpmu}, together with parts (a1) and (a2) applied to the Lagrange colleague pencil.
\end{proof}

\section{Strong linearizations for matrix polynomials in the Chebyshev basis}\label{sec:Chebyshev}

We finish the paper with the  Chebyshev bases. 
Some of the information that we include here can be found in \cite{lawrence-perez}, where an infinite family of block minimal basis linearizations of a matrix polynomial expressed in either the Chebyshev basis of the first kind or the second kind  is presented.

In order to write the results in a more compact way, we use a nonstandard notation to represent the Chebyshev polynomials.
 We denote by $\phi_n^{(1)}$ (resp. $\phi_n^{(2)}(\lambda)$) the $n$th Chebyshev polynomial of the first kind (resp. of the second kind). 
Our goal is, then, to construct strong linearizations for matrix polynomials of the form 
\begin{equation}\label{eq:matrix poly in Cheby bases}
P(\lambda) = \sum_{i=0}^{k} P_i\,\phi_i^{(r)}(\lambda), \quad P_0,\hdots,P_k\in\mathbb{C}^{n\times n}, \quad r\in\{1,2\}.
\end{equation}

Let $0\leq \epsilon \leq k-1$ be an integer, and let $n$ and $m$ be positive integers.
We define the matrix pencils
\begin{align}
&K_1^{(C,i)}(\lambda) = 
\begin{bmatrix}
	I_n & -2\lambda I_n & I_n \\
	& I_n & -2\lambda I_n & I_n \\
	& & \ddots & \ddots & \ddots \\
	& & & I_n & -2\lambda I_n & I_n \\
	& & & & I_n & -\phi_1^{(i)}(\lambda)I_n
\end{bmatrix}_{\epsilon n\times (\epsilon+1)n}, \label{eq:K1 Chebyshev}\\
&K_2^{(C,j)}(\lambda) = 
\begin{bmatrix}
	I_n & -2\lambda I_n & I_n \\
	& I_n & -2\lambda I_n & I_n \\
	& & \ddots & \ddots & \ddots \\
	& & & I_n & -2\lambda I_n & I_n \\
	& & & & I_n & -\phi_1^{(j)}(\lambda)I_n
\end{bmatrix}_{(k-1-\epsilon) m\times (k-\epsilon)m}, \label{eq:K2 Chebyshev}
\end{align}
where $i,j\in\{1,2\}$.
\begin{lemma}\label{lemma:dual min basis Chebyshev}
Let $0\leq \epsilon \leq k-1$ be an integer, and let $i,j\in\{1,2\}$.
The matrix pencils $K_1^{(C,1)}(\lambda)$ and $K_2^{(C,j)}(\lambda)$ given in \eqref{eq:K1 Chebyshev} and \eqref{eq:K2 Chebyshev} are both minimal bases.
Moreover, the matrix polynomials
\begin{equation}\label{eq:D1 and D2 Chebyshev}
D_1^{(C,i)}(\lambda) = 
\begin{bmatrix}
	\phi_\epsilon^{(i)}(\lambda)I_n \\
	\vdots \\
	\phi_1^{(i)}(\lambda)I_n \\[4pt]
	\phi_0^{(i)}(\lambda)I_n
\end{bmatrix} \quad \quad \mbox{and} \quad \quad
D_2^{(C,j)}(\lambda) = 
\begin{bmatrix}
	\phi_{k-1-\epsilon}^{(j)}(\lambda)I_m \\
	\vdots \\
	\phi_1^{(j)}(\lambda)I_m \\[4pt]
	\phi_0^{(j)}(\lambda)I_m
\end{bmatrix}
\end{equation}
are, respectively, dual minimal bases of $K_1^{(C,1)}(\lambda)$ and $K_2^{(C,j)}(\lambda)$.
\end{lemma}
\begin{proof}
The minimality of the four matrix polynomials follows readily from the characterization of minimal bases in Theorem \ref{minimal-basis}.
Moreover, by using the recurrence relationship of Chebyshev polynomials \eqref{eq:recurrence Cheby}, one can establish the duality by direct matrix multiplication.
\end{proof}
\begin{remark}
The reader might wonder why we use $\epsilon$ as parameter for the family of block minimal basis constructed in the previous definition instead of $\mu$, as we have done in the Newton and Lagrange case.
 We note that $\mu$ denoted the number of rows of the minimal basis $K_2(\lambda)$ in the block minimal basis pencils constructed in those two cases, while in the Chebyshev case it is more convenient to use the number of rows of $K_1(\lambda)$, that we denote by $\epsilon$. 
\end{remark}

We now consider strong block minimal basis pencils of the form
\begin{equation}\label{eq:Chebyshev pencils}
C(\lambda) = 
\begin{bmatrix}
	M(\lambda) & K_2^{(C,j)}(\lambda) \\
	K_1^{(C,i)}(\lambda) & 0 
\end{bmatrix}
\end{equation}
We will refer to \eqref{eq:Chebyshev pencils} as a \emph{Chebyshev pencil}.
The following theorem shows how to choose the body $M(\lambda)$ so that the Chebyshev pencil \eqref{eq:Chebyshev pencils} is a strong linearization of the matrix polynomial \eqref{eq:matrix poly in Cheby bases}.

\begin{theorem}\label{thm:colleague Chebyshev pencil}
Let $P(\lambda) = \sum_{i=0}^k P_i\,\phi_i^{(r)}(\lambda)$, where $r\in\{1,2\}$, be an $m\times n$ matrix polynomial expressed in a Chebyshev basis.
Let $0\leq \epsilon \leq k-1$ be an integer, and let
  \begin{equation*}\label{M Chebyshev}
M_{\epsilon}^C (\lambda):= \begin{bmatrix}
2\lambda P_{k} + P_{k-1} &  - P_{k} & 0 & \cdots & \cdots & 0 \\
P_{k-2} - P_{k} &  - P_{k-1} & \vdots & \ddots &  & \vdots \\
P_{k-3} &  \vdots & \vdots &  & \ddots & \vdots \\
\vdots&  - P_{\varepsilon + 2} & 0 & \cdots & \cdots & 0 \\
P_{\varepsilon} & P_{\varepsilon -1} - P_{\varepsilon + 1} & P_{\varepsilon -2} & P_{\varepsilon - 3} & \hdots & P_{0}
\end{bmatrix},
\end{equation*}
when $1 \leq \epsilon \leq k-2$; 
\begin{equation*}
M_{\epsilon}^C(\lambda):=\begin{bmatrix}
2\lambda P_k + P_{k-1} & P_{k-2}-P_{k} & P_{k-3} & \dots & P_{1} & P_{0}
\end{bmatrix},
\end{equation*}
when $\epsilon= k-1$; 
\begin{equation*}
M_{\epsilon}^C(\lambda):=\frac{1}{2} \begin{bmatrix} 2 \lambda P_k + P_{k-1} & P_{k-2}-2 P_k & P_{k-3} - P_{k-1} & \cdots & P_1 - P_3 & 2 P_0 - P_2  \end{bmatrix}^\mathcal{B},
\end{equation*}
when $\epsilon= 0$ and $r=1$; and
\begin{equation*}
M_{\epsilon}^C(\lambda):=\begin{bmatrix}
2\lambda P_k + P_{k-1} & P_{k-2}-P_{k} & P_{k-3} & \cdots & P_{1} & P_{0}
\end{bmatrix}^\mathcal{B},
\end{equation*}
when $\epsilon=0$ and $r=2$.
\begin{itemize}
\item[\rm(a)] If $P(\lambda)$ is expressed in the Chebyshev basis of the first kind, then the Chebyshev pencil
\begin{equation}\label{eq:colleague Cheby 1}
C_P^\epsilon(\lambda) = 
\begin{bmatrix}
	M_{\epsilon}^C(\lambda) & K_2^{(C,2)}(\lambda)^T\\[4pt]
	K_1^{(C,1)}(\lambda) & 0
\end{bmatrix}
\end{equation}
is a strong linearization of $P(\lambda)$. 
\item[\rm(b)]
If $P(\lambda)$ is expressed in the Chebyshev basis of the second kind, then the Chebyshev pencil
\begin{equation}\label{eq:colleague Cheby 2}
C_P^\epsilon(\lambda) = 
\begin{bmatrix}
	M_{\epsilon}^C(\lambda) & K_2^{(C,2)}(\lambda)^T\\[4pt]
	K_1^{(C,2)}(\lambda) & 0
\end{bmatrix}
\end{equation}
is a strong linearization of $P(\lambda)$.
\end{itemize}
We will refer to \eqref{eq:colleague Cheby 1}-\eqref{eq:colleague Cheby 2} as the \emph{colleague Chebyshev pencil} of $P(\lambda)$ associated with the parameter $\epsilon$.
\end{theorem}
\begin{proof}
The proof follows by using Theorem \ref{thm:key1}, together with Lemmas  \ref{cheb-ident} and \ref{lemma:dual min basis Chebyshev}.
\end{proof}

\begin{remark}
In the case where the matrix polynomial $P(\lambda)$ is expressed in the Chebyshev polynomial basis of the first kind, one could consider a colleague pencil of the form
\[
C_P(\lambda) = 
\begin{bmatrix}
	M(\lambda) & K_2^{(C,1)}(\lambda)^T\\[4pt]
	K_1^{(C,2)}(\lambda) & 0
\end{bmatrix}.
\]
The construction of linearizations of this form is very similar to the case \eqref{eq:colleague Cheby 1}, so we do not pursue this further.
One could also consider a colleague pencil of the form 
\[
C_P(\lambda) = 
\begin{bmatrix}
	M(\lambda) & K_2^{(C,1)}(\lambda)^T\\[4pt]
	K_1^{(C,1)}(\lambda) & 0
\end{bmatrix}.
\]
However, when constructing linearizations of this form, some of the block entries of $M(\lambda)$ become  linear combinations  of a large number of matrix coefficients of $P(\lambda)$ and thus, may cause numerical problems due to cancellation errors; see, for example, \cite[Remark 3.8]{lawrence-perez}.
\end{remark}

\begin{example}
Let $P(\lambda)= \sum_{i=0}^5 P_i\, \phi_i^{(1)}(\lambda)$ be an $m\times n$ matrix polynomial of degree 5 expressed in the Chebyshev basis of the first kind. 
Let   $\epsilon=3$. 
Then,
\[
\mathcal{C}_P^{\epsilon}(\lambda) = \left[ \begin{array}{cccc|c} 2\lambda P_5 + P_4 & - P_5 & 0 & 0 & I_m \\ P_3- P_5 & P_2-P_4 & P_1 & P_0 & -2\lambda I_m \\
\hline
I_n & -2\lambda I_n & I_n & 0  & 0 \\
0 & I_n & -2\lambda I_n & I_n & 0\\
0 & 0 & I_n & -\lambda I_n & 0
\end{array}\right]
\]
is the colleague Chebyshev pencil of $P(\lambda)$ associated with $\epsilon=3$.

Let $P(\lambda)= \sum_{i=0}^5 P_i\, \phi_1^{(2)}(\lambda)$ be an $m\times n$ matrix polynomial of degree 5 expressed in the Chebyshev basis of the second kind.
Let  $\epsilon=1$. 
Then
\[
\mathcal{C}_P^{\epsilon}(\lambda) = \left[\begin{array}{cc|ccc} 2\lambda P_5 + P_4 & - P_5 &  I_m & 0 & 0 \\
P_3-P_5 & - P_4 & -2\lambda I_m & I_m & 0 \\
P_2 & -P_3 & I_m & -2\lambda I_n  & I_m \\
P_1 & P_0-P_2 & 0 & I_m & -2\lambda I_m \\
\hline
I_n & -2\lambda I_n & 0 & 0 & 0
\end{array} \right]
\]
is the colleague Chebyshev pencil of $P(\lambda)$ associated with $\epsilon=1$.
\end{example}

\begin{remark}
A drawback of the Chebyshev  colleague linearizations of a matrix polynomial $P(\lambda)$ is that they are not companion forms since the matrix coefficient corresponding to the zero-degree term of these linearizations contains blocks which are sums of matrix coefficients of $P(\lambda)$. The Newton and Lagrange colleague linearizations are companion forms though. 
\end{remark}

An infinite family of  linearizations for matrix polynomials in the Chebyshev basis (of the first kind or the second kind) can be constructed combining the colleague Chebyshev pencil and Theorem \ref{thm:key2}.
\begin{theorem}
Let $P(\lambda)=\sum_{i=0}^k P_i \, \phi_i^{(r)}(\lambda)$, where $r\in\{1,2\}$, be an $m\times n$ matrix polynomial expressed in a Chebyshev basis.
Let $0\leq \epsilon\leq k-1$ be an integer and let $M_\epsilon^C(\lambda)$ be as in Theorem \ref{thm:colleague Chebyshev pencil}.
Let $A$ and $B$ be two arbitrary matrices of sizes $(k-\epsilon)m\times \epsilon n$ and $(k-1-\epsilon)m\times (\epsilon+1)n$, respectively.
\begin{itemize}
\item[\rm(a)] If $r=1$, then the Chebyshev pencil
\begin{equation}\label{eq:Cheby linearization 1}
C(\lambda) = 
\begin{bmatrix}
	M_{\epsilon}^C(\lambda)+AK_1^{(C,1)}(\lambda)+ K_2^{(C,2)}(\lambda)^TB& K_2^{(C,2)}(\lambda)^T\\[4pt]
	K_1^{(C,1)}(\lambda) & 0
\end{bmatrix}
\end{equation}
is a strong linearization of $P(\lambda)$.
\item[\rm(b)]
If $r=2$, then the Chebyshev pencil
\begin{equation}\label{eq:Cheby linearization 2}
C(\lambda) = 
\begin{bmatrix}
	M_{\epsilon}^C(\lambda)+AK_1^{(C,2)}(\lambda)+ K_2^{(C,2)}(\lambda)^TB & K_2^{(C,2)}(\lambda)^T\\[4pt]
	K_1^{(C,2)}(\lambda) & 0
\end{bmatrix}
\end{equation}
is a strong linearization of $P(\lambda)$.
\end{itemize}
We will refer to a Chebyshev pencil of the form \eqref{eq:Cheby linearization 1}-\eqref{eq:Cheby linearization 2} as a \emph{Chebyshev linearization} of $P(\lambda)$.
\end{theorem}
\begin{remark}
Observe that every Chebyshev linearization \eqref{eq:Cheby linearization 1}-\eqref{eq:Cheby linearization 2} is strictly equivalent to the colleague pencil \eqref{eq:colleague Cheby 1}-\eqref{eq:colleague Cheby 2}:
\begin{equation}\label{eq:Chebyshev equivalence}
C(\lambda) = 
\begin{bmatrix}
	I_{(k-\epsilon)m} & A \\
	0 & I_{\epsilon n}
\end{bmatrix}
C_P^\epsilon(\lambda)
\begin{bmatrix}
	I_{(\epsilon+1)n} & 0 \\
	B & I_{(k-1-\epsilon)m}
\end{bmatrix}.
\end{equation}
\end{remark}

In the following two sections, we obtain recovery rules for eigenvectors, and minimal bases and minimal indices of a matrix polynomial $P(\lambda)$ from those of its Chebyshev linearizations.
We will need the following definitions and results.
\begin{definition}[Chebyshev-Horner shifts]
Let $k$ and $0\leq \epsilon \leq k-1$ be integers.
Given a matrix polynomial $P(\lambda)= \sum_{i=0}^k P_i\, \phi_i^{(r)}(\lambda)$ expressed in the Chebyshev basis of the  $r$th kind, where $r\in\{1, 2\}$, the \emph{$i$th Chebyshev-Horner shift} of $P(\lambda)$ associated with $ \epsilon$ is given by
\[
P_{\epsilon,r}^i(\lambda):= P_k \phi^{(r)}_{\epsilon+i}(\lambda)+ P_{k-1}  \phi^{(r)}_{\epsilon+i-1}(\lambda)+ \cdots+ P_{k-i+1} \phi^{(r)}_{\epsilon+1}(\lambda) + P_{k-i} \phi^{(r)}_{\epsilon}(\lambda),
\]
for $i=0,1,\hdots,k- \epsilon$.
Note that $P_{0,r}^0(\lambda) = P_k$ and $P_{0,r}^k(\lambda)= P(\lambda)$, for $r=1,2$.
\end{definition}

Lemma \ref{horner-prop-cheb} provides a property of the Chebyshev-Horner shifts of a matrix polynomial that will be useful to prove Theorem \ref{thm:one-sided-chev}. 
\begin{lemma}\label{horner-prop-cheb}
Let $P(\lambda) = \sum_{i=0}^k P_i \,\phi^{(r)}_i(\lambda)$, with $r\in \{1,2\}$,  be a matrix polynomial of degree $k$ expressed in the Chebyshev  basis of $r$th kind. 
Then, the $i$th Chebyshev Horner shift polynomial $P_{\epsilon,r}^i(\lambda)$ is a polynomial of degree $\epsilon + i$ and 
\[
P_{\epsilon,r}^{i+1}(\lambda) =  2 \lambda P_{\epsilon,r}^{i}(\lambda)- P_{\epsilon-1,r}^i(\lambda)+ P_{k-i-1} \phi^{(r)}_{\epsilon}(\lambda) \quad (i=1,\hdots,k-1).
\]
\end{lemma}
\begin{proof}
From $\phi_j^{(r)}(\lambda) = 2 \lambda \phi_{j-1}^{(r)}(\lambda) - \phi_{j-2}^{(r)}(\lambda)$, $r\in \{1, 2\}$, we obtain $2 \lambda P_{\epsilon,r}^i(\lambda) - P_{\epsilon-1,r}^i(\lambda)  = P_k \phi^{(r)}_{\epsilon+i+1}(\lambda) + P_{k-1}\phi^{(r)}_{\epsilon+i}(\lambda) + \cdots + P_{k-i} \phi^{(r)}_{\epsilon+1}(\lambda)$
The result  now follows from the definition of Chebyshev Horner shift of $P(\lambda)$ and the fact that the Chebyshev bases are degree-graded bases.
\end{proof}

Theorem \ref{thm:one-sided-chev} gives right- and left-sided factorizations of the colleague Chebyshev pencil \eqref{eq:colleague Cheby 1}-\eqref{eq:colleague Cheby 2}.
\begin{theorem}\label{thm:one-sided-chev}
Let $P(\lambda)= \sum_{i=0}^k P_i \, \phi_i^{(r)} (\lambda)$, where $r\in \{1, 2\}$, be a matrix polynomial expressed in  the Chebyshev basis of the $r$th kind.
 Let $0\leq \epsilon \leq k-1$ be an integer, let $C_P^{r,\epsilon}(\lambda)$ be the colleague Chebyshev pencil \eqref{eq:colleague Cheby 1}-\eqref{eq:colleague Cheby 2} of $P(\lambda)$ associated with $\epsilon$, and let $D_1^{(C,i)}(\lambda)$ and $D_2^{(C,j)}(\lambda)$ be the minimal bases defined in \eqref{eq:D1 and D2 Chebyshev}.

For $0 < \epsilon < k-1$ and $r\in \{1,2\}$, define
\[
H_C^{\epsilon}(\lambda)^{\mathcal{B}}:=\begin{bmatrix} D_1^{(C,r)}(\lambda) & -P_{\epsilon,r}^1(\lambda)& -P_{\epsilon,r}^2(\lambda) & \cdots & -P_{\epsilon,r}^{k-\epsilon-1}(\lambda)\end{bmatrix} 
\]
and
\[
G_C^{\epsilon}(\lambda)^{\mathcal{B}}:=\begin{bmatrix} D_2^{(C,2)}(\lambda) & -P_{0,2}^{k-\epsilon}(\lambda)& -P_{0,2}^{k-\epsilon+1}(\lambda) & \cdots & -P_{0,2}^{k-1}(\lambda) \end{bmatrix}.
\]

For $\epsilon =0$ and $r=1$, define
\[
H_C^{ \epsilon}(\lambda)^{\mathcal{B}} := 
\begin{bmatrix} 
	I_n & -P_{0,1}^1(\lambda) + \dfrac{P_{k-1}}{2} & -P_{0,1}^2(\lambda) +     \dfrac{P_{k-2}}{2}  & \cdots & -P_{0,1}^{k-1}(\lambda) + \dfrac{P_1}{2} \end{bmatrix}
\]
and $G_C^{\epsilon}(\lambda)^{\mathcal{B}}:=D_2^{(C,2)}(\lambda)$.

For $\epsilon =0$ and $r=2$, define
\[H_C^{\epsilon}(\lambda)^{\mathcal{B}} := 
\begin{bmatrix} I_n & -P_{0,2}^1(\lambda) & -P_{0,2}^2(\lambda) & \cdots & -P_{0,2}^{k-1}(\lambda) \end{bmatrix}
\]
and $G_C^{\epsilon}(\lambda)^{\mathcal{B}}:=D_2^{(C,2)}(\lambda)$.

For $\epsilon = k-1$ and $r\in \{1,2\}$, define $H_C^{r, \epsilon}(\lambda)^{\mathcal{B}} :=  D_1^{(C,r)}(\lambda)$ and
\[
G_C^{\epsilon}(\lambda)^{\mathcal{B}}:= 
\begin{bmatrix} I_n &  -P_{0, 2}^1(\lambda) & -P_{0,2}^2(\lambda) & \cdots & -P_{0,2}^{k-1}(\lambda)
\end{bmatrix}.
\]

Then, the following right- and left-sided factorizations hold
\[
\mathcal{C}_P^{\epsilon}(\lambda) H_C^{\epsilon}(\lambda) = e_{k-\epsilon} \otimes P(\lambda), \quad \mbox{and} \quad G_C^{\epsilon}(\lambda)^{\mathcal{B}}\mathcal{C}_P^{ \epsilon}(\lambda)  = e^T_{\epsilon+1} \otimes P(\lambda),
\]
where $e_i$ denotes the $i$th column of the $k\times k$ identity matrix.
\end{theorem}
\begin{proof}
By using Lemma \ref{horner-prop-cheb}, the results can be easily shown  using straightforward but tedious calculations.
\end{proof}

\subsection{Recovery of eigenvectors from Chebyshev linearizations}
Assume that the matrix polynomial $P(\lambda)$ is regular.
In this section, we show how to recover (left and right) eigenvectors of $P(\lambda)$  from those of its Chebyshev linearizations.

First, Theorem \ref{thm:eigenvector formula Chebyshev} gives a close formula for the right and left eigenvectors of the Chebyshev pencil \eqref{eq:colleague Cheby 1}-\eqref{eq:colleague Cheby 2} associated with its finite eigenvalues.
\begin{theorem}\label{thm:eigenvector formula Chebyshev}
Let $P(\lambda)=\sum_{i=0}^k P_i\,\phi_i^{(r)}(\lambda)$ be an $n \times n$ regular matrix polynomial expressed in the Chebyshev basis of $r$th kind, where $r\in\{1, 2\}$. 
Let $\lambda_0$ be a finite eigenvalue of $P(\lambda)$.
Let $0 \leq \epsilon \leq k-1$ be an integer and let  $\mathcal{C}_{P}^{\epsilon}(\lambda)$ be the Chebyshev colleague pencil of $P(\lambda)$ associated with $\epsilon$ (defined in \eqref{eq:colleague Cheby 1}-\eqref{eq:colleague Cheby 2}). 
 Then, $z$ (resp. $w$) is a right (resp. left) eigenvector of $\mathcal{C}_{P}^{\epsilon}(\lambda)$ associated with $\lambda_0$ if and only if $z=  H_C^{\epsilon}(\lambda_0) x$ (resp. $w = G_C^{\epsilon}(\lambda_0) y$), where $x$ (resp. $y$) is a right (resp. left) eigenvector of $P(\lambda)$ with eigenvalue $\lambda_0$. 
\end{theorem}
\begin{proof}
This result is an immediate consequence of Theorems \ref{thm:factorizations} and \ref{thm:one-sided-chev}.
\end{proof}

Theorem \ref{thm:recovery eigenvectors Chebyshev} shows how to recover the eigenvectors of the matrix polynomial $P(\lambda)$ from those of its Chebyshev linearizations.
\begin{theorem}[Recovery of eigenvectors from Chebyshev linearizations]\label{thm:recovery eigenvectors Chebyshev}
Let $P(\lambda)=\sum_{i=0}^k P_i\,\phi_i^{(r)}(\lambda)$ be an $n \times n$ regular matrix polynomial expressed in the Chebyshev basis of $r$th kind, where $r\in\{1, 2\}$, and let $\lambda_0$ be an eigenvalue (finite or infinite) of $P(\lambda)$.
Let $C(\lambda)$ be a Chebyshev linearization of $P(\lambda)$ as in \eqref{eq:Cheby linearization 1}-\eqref{eq:Cheby linearization 2}.
Let $z$ and $\omega$ be, respectively, a right and a left eigenvector of $C(\lambda)$ associated with $\lambda_0$.
\begin{enumerate}
\item Assume $\lambda_0$ is finite. 
Then,
\begin{itemize}
\item the block entry $z(\epsilon+1)$ is a right eigenvector of $P(\lambda)$ with eigenvalue $\lambda_0$, and
\item the  block entry $w(k-\epsilon)$ is a left eigenvector of $P(\lambda)$ with eigenvalue $\lambda_0$.
\end{itemize}
\item Assume $\lambda_0$ is infinite.
Then,
\begin{itemize}
\item the block entry $z(1)$ is a right eigenvector of $P(\lambda)$ with eigenvalue at infinity, and
\item the block entry $w(1)$ is a left eigenvector of $P(\lambda)$ with eigenvalue at infinity.
\end{itemize}
\end{enumerate}
\end{theorem}
\begin{proof}
We prove the result for the right eigenvectors.
The proof for the left eigenvectors is analogous.

We first show that the theorem holds for the Chebyshev colleague pencil $C_P^\epsilon(\lambda)$ defined in \eqref{eq:colleague Cheby 1}-\eqref{eq:colleague Cheby 2}.

Case I: Assume that $\lambda_0$ is a finite eigenvalue, and let $z$ be a right eigenvector of the Chebyshev colleague pencil associated with $\lambda_0$.
From Theorem \ref{thm:eigenvector formula Chebyshev}, we obtain that $z=H_C^\epsilon(\lambda_0)x$, for some right eigenvector $x$ of $P(\lambda)$ with eigenvalue $\lambda_0$.
Then, the recovery rule follows from the fact that the $\epsilon+1$ block-entry of $H_C^\epsilon(\lambda_0)x$ is the vector $x$.

Case II: Assume that $\lambda_0$ is an infinite eigenvalue.
Since the Chebyshev bases are degree-graded, we have that $\mathrm{rev}_k\,P(0)=P_k$.
Hence, $x$ is an eigenvector  of $P(\lambda)$ with eigenvalue at infinity if and only if $x\neq 0$ and $P_kx=0$.
Moreover, if $0<\epsilon<k-1$, by evaluating the reversal of the Chebyshev colleague pencil at $\lambda=0$, we obtain
\[
\textrm{rev}_1\, \mathcal{C}_P^{\epsilon} (0) = \begin{bmatrix} 2 P_k & 0 & 0 & \cdots & 0 & 0 & 0 &  \cdots & 0 \\
0 & 0 & 0 & \cdots & 0 & -2 I_n & 0 &  \cdots & 0\\
0 & 0 & 0 & \cdots & 0 & 0 & -2 I_n & \cdots & 0\\
\vdots & \vdots & \vdots & \ddots & \vdots & \vdots & \vdots & \ddots & \vdots \\
0 & 0 & 0 & \cdots & 0 & 0 & 0 &  \cdots & -2 I_n\\
0 & -2 I_n & 0 & \cdots & 0 & 0 & 0  & \cdots & 0\\
0 &  0 & -2 I_n  & \cdots & 0 & 0 & 0  & \cdots & 0\\
\vdots & \vdots & \vdots & \ddots & \vdots & \vdots & \vdots & \ddots & \vdots \\
0 & 0 & 0 & \cdots & - r I_n & 0 & 0 & \cdots & 0
\end{bmatrix}.
\]
Thus, every right  eigenvector $z$ of $\mathcal{C}_P^{\epsilon}(\lambda)$ with eigenvalue at infinity must be of the form $\begin{bmatrix} x & 0 & \cdots & 0 \end{bmatrix}$, for some right eigenvector $x$ of $P(\lambda)$ with eigenvalue at infinity. 
A similar argument shows that this is also the case when $\epsilon=0$ or $\epsilon = k-1$. 

The recovery rules when $C(\lambda)$ is a Chebyshev linearization other than the Chebyshev colleague pencil follow from the Chebyshev colleague's recovery rules and the equivalence transformation in \eqref{eq:Chebyshev equivalence}.
\end{proof}

\subsection{Recovery of minimal bases and minimal  indices from Chebyshev linearizations}
Assume that the matrix polynomial $P(\lambda)$ is singular.
In this section, we show how to recover the minimal indices and minimal bases of $P(\lambda)$ from those of its Chebyshev linearizations.

\begin{theorem}[Recovery of minimal bases and minimal indices from Chebyshev linearizations]\label{minimal-cheb}{\rm \cite{lawrence-perez}}
Let $P(\lambda)=\sum_{i=0}^k P_i\,\phi_i^{(r)}(\lambda)$ be an $m \times n$ singular matrix polynomial expressed in the Chebyshev basis of $r$th kind, where $r\in\{1, 2\}$, and let $0 \leq \epsilon \leq k-1$ be an integer.
Let $C(\lambda)$ be a Chebyshev linearization of $P(\lambda)$ as in \eqref{eq:Cheby linearization 1}-\eqref{eq:Cheby linearization 2}.
\begin{enumerate}
\item[\rm(a1)] Suppose that $\{z_1(\lambda), \ldots, z_p(\lambda)\}$  is any right minimal basis of $C(\lambda)$, with vectors partitioned into blocks conformable to the blocks of $C(\lambda)$, and let $x_{\ell}(\lambda)$ be the $(\epsilon+1)$th block of $z_{\ell}(\lambda)$, for $\ell = 1, 2, \ldots, p$. 
Then, $\{x_1(\lambda), \hdots, x_p(\lambda)\}$ is a right minimal basis of $P(\lambda)$.
\item[\rm(a2)] If $0 \leq \epsilon_1 \leq \cdots \leq \epsilon_p$ are the right minimal indices of $C(\lambda)$, then
\[
0 \leq \epsilon_1- \epsilon \leq \epsilon_2- \epsilon \leq \cdots \leq \epsilon_p - \epsilon
\]
are the right minimal indices of $P(\lambda)$.
\item[\rm(b1)] Suppose that $\{w_1(\lambda), \cdots, w_q(\lambda)\}$  is any left minimal basis of $C(\lambda)$, with vectors partitioned into blocks conformable to the blocks of $C(\lambda)$, and let $y_{\ell}(\lambda)$ be the $(k-\epsilon)$th block of $w_{\ell}(\lambda)$, for $\ell = 1, 2, \hdots, q$. 
Then, $\{y_1(\lambda), \ldots, y_q(\lambda)\}$ is a left minimal basis of $P(\lambda)$.
\item[\rm(b2)] If $0 \leq \mu_1 \leq \cdots \leq \mu_q$ are the left minimal indices of $C(\lambda)$, then
\[
0 \leq \mu_1 - k + 1+ \epsilon  \leq \epsilon_2- k+ 1+\epsilon \leq \cdots \leq \epsilon_p - k+1+ \epsilon
\]
are the left minimal indices of $P(\lambda)$.
\end{enumerate}
\end{theorem}

\section{Conclusions}
When solving a polynomial eigenvalue problem (PEP) $P(\lambda)x=0$, the polynomial $P(\lambda)$ is sometimes expressed in a basis other than the monomial basis, for example, when it is the approximation of a nonlinear eigenvalue problem. In particular, the Chebyshev, Newton and Lagrange bases are the most commonly used.  The solution of a PEP usually involves a linearization. In the literature, most of the available linearizations are constructed from the coefficients of the polynomial expressed in the monomial basis. From the numerical point of view, it is not wise to do the computations necessary to express  $P(\lambda)$ in the monomial basis, when it is originally expressed in a non-monomial basis,  in order to use one of the linearizations in the literature. A much better approach is to construct linearizations that can directly be constructed from  the matrix coefficients of $P(\lambda)$ regardless of the basis it is expressed in.  In this paper, we have constructed three families of block minimal basis pencils that are strong linearizations of $P(\lambda)$ when it is expressed in one of the three non-monomial bases mentioned above. These linearizations are easy to construct from the coefficients of $P(\lambda)$ and they include the so-called ``colleague linearizations'' for each type of basis used in the literature. Additionally, we have shown that it is easy to recover the eigenvectors, minimal bases and minimal indices of $P(\lambda)$ from those of the linearizations. We notice though that not all of the families are equally convenient when  solving a nonlinear eigenvalue problem $T(\lambda)x=0$. While the Newton and Lagrange bases can be used when the domain of $T$ is a subset of the complex numbers, the Chebyshev basis can only be used when the domain of $T$ is a subset of the real numbers or a  parametrizable curve. Moreover, the linearizations that we construct as well as the few available in the literature are companion forms in the Newton and Lagrange case while those in the Chebyshev family are not. However, the Chebyshev basis is the most commonly used basis in these applications. Our goal, in a subsequent paper, is to compare the linearizations in these three families from the numerical point of view, that is, in terms of conditioning of eigenvalues and backward errors with the objective of providing a guidance on what bases to use in each situation and, once chosen a basis, provide information about what linearization, within the family, has a better performance.

\appendix

\section{Proof of Theorem \ref{thm:factorizations}}

Parts (e) and (f) have been proven in \cite[Theorem 3.6]{Linearizations}.
Moreover, parts (b) and (d) follow from applying parts (a) and (c) to $L(\lambda)^T$ and $P(\lambda)^T$ and then taking transposes.
Hence, we only need to prove parts (a) and (c).

\medskip

\noindent {\bf Proof of part (a):} Let $\lambda_0$ be a finite eigenvalue of $P(\lambda)$ and let $g:=\mathrm{dim}\,\mathcal{N}_r(P(\lambda_0))$.
 Since $L(\lambda)$ is a strong linearization of $P(\lambda)$, we have that $\lambda_0$ is an eigenvalue of $L(\lambda)$ and  $\mathrm{dim}\,\mathcal{N}_r(L(\lambda_0))=g$.

Let $\{x_1,\hdots,x_g\}$ be a basis for $\mathcal{N}_r(P(\lambda_0))$, and consider the vectors
\[
v_i=
\begin{bmatrix}
	D_1(\lambda_0)^T\\
	X(\lambda_0)
\end{bmatrix}x_i \quad (i=1,\hdots,g).
\]
We are going to prove that $\{v_1,\hdots,v_g\}$ is a basis for $\mathcal{N}_r(L(\lambda_0))$.
First, we note that vectors $v_i$ are nonzero because $D_1(\lambda)^T$ has full column rank for any $\lambda\in\mathbb{C}$ since it is a minimal basis.
Second, from the right-sided factorization, we get
\[
L(\lambda_0)v_i = L(\lambda_0)
\begin{bmatrix}
	D_1(\lambda_0)^T\\
	X(\lambda_0)
\end{bmatrix}x_i = (v\otimes I_n)P(\lambda)x_i = 0.
\]
Hence, $v_i\in\mathcal{N}_r(L(\lambda_0))$.
To finish the proof, it suffices to show that the vectors $v_i$ are linearly independent.
Assume they are not independent, that is, assume there are constants $c_i$, not all zero, such that $c_1v_1+\cdots+c_pv_p=0$.
Then,
\[
0 = c_1v_1+\cdots+c_pv_p =
\begin{bmatrix}
	D_1(\lambda_0)^T\\
	X(\lambda_0)
\end{bmatrix}(c_1x_1+\cdots+c_px_p),
\]
which implies $c_1x_1+\cdots+c_px_p=0$.
But this contradicts the fact that the $x_i$ vectors are linearly independent. 
Thus, the vectors $v_i$ must be independent and form a basis for $\mathcal{N}_r(L(\lambda_0))$.

Let $\{v_1,\hdots,v_g\}$ be a basis for $\mathcal{N}_r(L(\lambda_0))$.
We are going to show that
\[
v_i=
\begin{bmatrix}
	D_1(\lambda_0)^T\\
	X(\lambda_0)
\end{bmatrix}x_i \quad (i=1,\hdots,p),
\]
for some basis $\{x_1,\hdots,x_p\}$ of $\mathcal{N}_r(P(\lambda_0))$.
Let $\{\widetilde{x}_1,\hdots,\widetilde{x}_p\}$ be some basis for $\mathcal{N}_r(P(\lambda_0))$.
Then, we have that
\[
\left\{ 
\widetilde{v}_1:=\begin{bmatrix}
	D_1(\lambda_0)^T\\
	X(\lambda_0)
\end{bmatrix}\widetilde{x}_1,
\hdots,
\widetilde{v}_p:=\begin{bmatrix}
	D_1(\lambda_0)^T\\
	X(\lambda_0)
\end{bmatrix}\widetilde{x}_p
\right\}
\]
is a basis for $\mathcal{N}_r(L(\lambda_0))$,  as proven above.
Hence
\[
v_i = \sum_{j=1}^pc_j^{(i)}\widetilde{v}_i = 
\begin{bmatrix}
	D_1(\lambda_0)^T\\
	X(\lambda_0)
\end{bmatrix}\sum_{j=1}^pc_j^{(i)}\widetilde{x}_i =:
\begin{bmatrix}
	D_1(\lambda_0)^T\\
	X(\lambda_0)
\end{bmatrix}x_i \quad (i=1,\hdots,p),
\]
for some constants $c_j^{(i)}$. 
To finish the proof, it suffices to show that the vectors $x_i\in\mathcal{N}_r(P(\lambda_0))$ are linearly independent.
But their independence follows easily from the fact that the $v_i$ vectors are independent.

\medskip

\noindent {\bf Proof of part (c):}  
Since $L(\lambda)$ is a strong linearization of $P(\lambda)$, we have $p:=\mathrm{dim}\,\mathcal{N}_r(P(\lambda))=\mathrm{dim}\,\mathcal{N}_r(L(\lambda))$.

Let $\{x_1(\lambda),\hdots,x_p(\lambda)\}$ be a minimal basis of $\mathcal{N}_r(P(\lambda))$ and let $\epsilon_i:=\mathrm{deg}\,x_i(\lambda)$, for $i=1,\hdots,p$.
Without loss of generality, assume $\epsilon_1\geq \epsilon_2\geq \cdots \geq \epsilon_p$.
Consider the polynomial vectors
\[
v_i(\lambda)=
\begin{bmatrix}
	D_1(\lambda)^T\\
	X(\lambda)
\end{bmatrix}x_i(\lambda) \quad (i=1,\hdots,p).
\] 
From the right-sided factorization, we obtain
\[
L(\lambda)v_i(\lambda) =
L(\lambda) \begin{bmatrix}
	D_1(\lambda)^T\\
	X(\lambda)
\end{bmatrix}x_i(\lambda) =
(v\otimes I_m)P(\lambda)x_i(\lambda)=0.
\]
Thus $v_i(\lambda)\in \mathcal{N}_r(L(\lambda))$, for $i=1,\hdots,p$.
Furthermore, the polynomial vectors $v_i(\lambda)$ are linearly independent because the polynomial vectors $x_i(\lambda)$ are independent and $D_1(\lambda)^T$ has full column rank.
Hence, according to part (e), to show that $\{v_1(\lambda),\hdots,v_p(\lambda)\}$ is a basis for $\mathcal{N}_r(L(\lambda))$, it suffices to show that $\mathrm{deg}\, v_i(\lambda) = \epsilon_i + \mathrm{deg}\,D_1(\lambda)$, for $i=1,\hdots,p$.
This degree shifting property follows from the following argument.
From $L(\lambda)v_i(\lambda)=0$, we get 
\begin{equation}\label{K2M}
K_2(\lambda)^TX(\lambda)x_i(\lambda)= -M(\lambda)D_1^T(\lambda)x_i(\lambda).
\end{equation}
We note that 
\begin{align*}
\mathrm{deg}\, K_2(\lambda)^TX(\lambda)x_i(\lambda) = \mathrm{deg}\,K_2(\lambda)^T + \mathrm{deg}\,X(\lambda)x_i(\lambda) =& 1+ \mathrm{deg}\,X(\lambda)x_i(\lambda). 
\end{align*}
where the first equality follows from the fact that  $K_2(\lambda)$ is a minimal basis.
Moreover,   $\mathrm{deg} M(\lambda)D_1^T(\lambda) x_i(\lambda) \leq 1+\mathrm{deg} D_1^T(\lambda)x_i(\lambda)$. 
Then, by \eqref{K2M}, we get $\mathrm{deg}\,X(\lambda)x_i(\lambda) \leq  \mathrm{deg}\,D_1(\lambda)x_i(\lambda)$ for $i=1,\hdots,p$.
Therefore,
\begin{align}\label{eq:deg shift}
\begin{split}
\mathrm{deg}\, v_i(\lambda) =& \mathrm{deg}\begin{bmatrix}
	D_1(\lambda)^Tx_i(\lambda)\\
	X(\lambda)x_i(\lambda)
\end{bmatrix} = \max\{\mathrm{deg} D_1(\lambda)^Tx_i(\lambda), \mathrm{deg} X(\lambda)x_i(\lambda)\} = \\
&\mathrm{deg}\,D_1(\lambda)^Tx_i(\lambda) = \mathrm{deg}\,x_i(\lambda)+\mathrm{deg}\,D_1(\lambda) = \epsilon_i + \mathrm{deg}\,D_1(\lambda),
\end{split}
\end{align}
where the fourth equality follows from the fact that $D_1(\lambda)$ is a minimal basis. 
This proves the claim.

Now we prove the converse. 
Let $\{v_1(\lambda),\hdots,v_p(\lambda)\}$ be a minimal basis for $\mathcal{N}_r(L(\lambda))$ ordered so that  $\mathrm{deg}\, v_1(\lambda)\geq \cdots \geq \mathrm{deg}\,v_p(\lambda)$.
We are going to show that
\[
v_i(\lambda)=
\begin{bmatrix}
	D_1(\lambda)^T\\
	X(\lambda)
\end{bmatrix}x_i(\lambda) \quad (i=1,\hdots,p),
\]
for some minimal basis $\{x_1(\lambda),\hdots,x_p(\lambda)\}$ of $\mathcal{N}_r(P(\lambda))$.
Let $\{\widetilde{x}_1(\lambda),\hdots,\widetilde{x}_p(\lambda)\}$ be some minimal basis for $\mathcal{N}_r(P(\lambda))$.
Then,  by the previous proof of part (c)  we have that
\[
\left\{ 
\widetilde{v}_1(\lambda):=\begin{bmatrix}
	D_1(\lambda)^T\\
	X(\lambda)
\end{bmatrix}\widetilde{x}_1(\lambda),
\hdots,
\widetilde{v}_p(\lambda):=\begin{bmatrix}
	D_1(\lambda)^T\\
	X(\lambda)
\end{bmatrix}\widetilde{x}_p(\lambda)
\right\}
\]
is a minimal basis for $\mathcal{N}_r(L(\lambda))$.
Hence
\[
v_i(\lambda) = \sum_{j=1}^pc_j^{(i)}(\lambda)\widetilde{v}_i(\lambda) = 
\begin{bmatrix}
	D_1(\lambda)^T\\
	X(\lambda)
\end{bmatrix}\sum_{j=1}^pc_j^{(i)}(\lambda)\widetilde{x}_i(\lambda) =:
\begin{bmatrix}
	D_1(\lambda)^T\\
	X(\lambda)
\end{bmatrix}x_i(\lambda) \quad (i=1,\hdots,p),
\]
for some (scalar) polynomials $c_j^{(i)}(\lambda)$ (see \cite{Forney}, Part 4 in Main Theorem). 
We observe that the polynomial vectors $x_i(\lambda)\in\mathcal{N}_r(P(\lambda))$ form a basis for $\mathcal{N}_r(P(\lambda))$, since they are linearly independent.
Moreover, the degree-shifting property \eqref{eq:deg shift} implies $\mathrm{deg}\,v_i(\lambda)=\mathrm{deg}\,D_1(\lambda)+\mathrm{deg}\, x_i(\lambda)$, and part (e) implies $\mathrm{deg}\,v_i(\lambda)=\epsilon_i + \mathrm{deg}\,D_1(\lambda)$, where $\epsilon_1,\hdots,\epsilon_p$ are the right minimal indices of $P(\lambda)$. 
Hence $\mathrm{deg}\, x_i(\lambda)=\epsilon_i$, for $i=1,\hdots,p$.
Therefore, $\{x_1(\lambda),\hdots,x_p(\lambda)\}$ is a minimal basis for $\mathcal{N}_r(P(\lambda))$.

\end{document}